\documentclass[10pt,reqno]{article}
\usepackage[usenames]{color}
\usepackage{amssymb}
\usepackage{amsmath}
\usepackage{amsthm}
\usepackage{amsfonts}
\usepackage{amscd}
\usepackage{url}
\usepackage{appendix}

\usepackage[colorlinks=true,
linkcolor=blue,
filecolor=webbrown,
citecolor=webgreen]{hyperref}

\definecolor{webgreen}{rgb}{0,.5,0}
\definecolor{webbrown}{rgb}{.6,0,0}
\usepackage{color}
\usepackage{fullpage}
\usepackage{float}

\usepackage{graphics}
\usepackage{latexsym}
\usepackage{epsf}

\numberwithin{equation}{section}

\theoremstyle{plain}
\newtheorem{thm}{\protect\theoremname}
\theoremstyle{plain}

\theoremstyle{definition}

\newtheorem{prop}{Proposition}[section]

\theoremstyle{plain}
\newtheorem*{prop*}{\protect\propositionname}
\theoremstyle{remark}

\newtheorem{rem}{Remark}[section]
\theoremstyle{plain}
\newtheorem{cor}{Corollary}[prop]

\makeatother

\usepackage{babel}

\providecommand{\propositionname}{Proposition}

\providecommand{\theoremname}{Theorem}

\newcommand{\seqnum}[1]{\href{https://oeis.org/#1}{\rm \underline{#1}}}

\begin{document}

\pagestyle{plain}
\numberwithin{equation}{section}
\title{An Extension of Glasser's Master Theorem and a Collection of Improper Integrals Many of Which Involve Riemann's Zeta Function}
\author{ Michael Milgram\footnote{mike@geometrics-unlimited.com}\\{Consulting Physicist, Geometrics Unlimited, Ltd.}}
\maketitle
\begin{flushleft} \vskip 0.3 in 
\centerline{Box 1484, Deep River, Ont. Canada. K0J 1P0}
\centerline{Author's manuscript, Orcid:0000-0002-7987-0820}
\vskip .2in
January 16, 2024: original\\
\today: minor revisions and corrections
\vskip .2in

MSC classes: 30E20; 33E20; 44A99	
\vskip 0.1in
Keywords: Analytic Integration, analysis, Master theorem, Integrals of Zeta function, Integrals of arbitrary functions

\vskip 0.1in 
\centerline{\bf Abstract}\vskip .3in

Glasser's Master Theorem \cite{GlasserMaster} is essentially a restatement of Cauchy's integral Theorem reduced to a specialized form. Here we extend that theorem by introducing two new parameters, but still retain a simple form.  Because of wide interest in entities involving Riemann's zeta function, the focus is on the evaluation of improper integrals with almost arbitrary integrands involving that function, but we also consider some other instances, the purpose being to demonstrate the power of the extended theorem. This is achieved by the presentation of a large number of examples that illustrate the ubiquity of the range of possibilities. One simple outcome of the study is the use of the extended theorem to show how it is possible to evaluate an integral when series or other representations of an integrand function do not converge. 

\section{Introduction} \label{sec:Intro}

Over the years, many ``Master Theorems" have been devised that allow the analytic integration of a large class of almost-arbitrary functions, the quintessential ones being those of Cauchy \cite[Chapter 5]{Knopp1} and Ramanujan \cite{RaMaster}, \cite{BradVig}. Here, we are interested in such a Theorem recently introduced by Glasser \cite{GlasserMaster} and studied by Glasser and Milgram \cite{Master}, \cite{2015arXiv150506515G}, where many specific and general examples were given. Considering the importance of the subject, it is worth noting that a survey of other ``Master" theorems was included in \cite{Master}; further citations to the literature on this subject can be found in \cite{ASQ}, \cite{CamAb}, \cite{SAQK} and \cite{ASQ2} where other, newly derived ``Master" theorems are discussed, and their variations explored.\newline

Along with the search for new methods of integrating general functions, interest in Riemann's zeta function $\zeta(s)$ continues unabated, with specific interest here focussed on improper integrals when $\zeta(s)$ is included in the integrand. Such evaluations are rare, perhaps because it is not always possible to ascertain whether such integrals converge numerically; in such cases one must fall back on the use of analytic continuation and regularization techniques borrowed from quantum field theory in order to attach meaning to such objects (see \cite[Appendix C]{Mdet}). However, in cases where convergence is not an issue, improper integrals are known \cite{Mdet}; further, proper integrals involving $\zeta(s)$ can be found listed in the literature \cite{HuKimKim},\cite{Esp&Moll2002} and \cite{Shpot&Paris}. Here we are interested in both extending Glasser's Master theorem and employing it to evaluate improper integrals that mostly involve $\zeta(s)$ and other functions. One of the interesting discoveries is a novel method of evaluating improper integrals when the series representation of an integrand function does not converge.

\section{The basics} \label{sec:Basics}
In Glasser's original work \cite{GlasserMaster}, a master integration identity
\begin{equation}
\int_{-\infty}^{\infty}\frac{F \! \left(a\,v \left(v +i \right)\right)}{\cosh \! \left(\pi \,v \right)}d v =F(a/4)
\label{Exf}
\end{equation}

was presented, which can be restated to read

\begin{equation}
\int_{-\infty}^{\infty}F \! \left(v \right)d v =-\pi i 
\sum_j\; \mathit{R} \! \left(F \! \left(v_{j} \right)\right)_{v_{ j}\in \mathfrak{S}}\,.
\label{Exf2}
\end{equation}

The validity of \eqref{Exf2} is subject to the condition that the integral exists and $F(v)$ is meromorphic in the strip

\begin{equation}
\mathfrak{S}\equiv -1< \Im(v) < 0\,,
\label{Sdef}
\end{equation}

where $\mathit{R} \! \left(F \! \left(v_{j} \right)\right)$ denotes the residue of $F(v_{j})$ for each point $v_{j}$ that lies in $\mathfrak{S}$, provided also that $F(v)$ satisfies the transformation 
\begin{equation}
F(v)+F(-v-i)=0\,.
\label{InvDef}
\end{equation}

The principle upon which that result was based is here generalized, with $0<b\in\Re$, by extending the definition of $\mathfrak{S}$ as follows:
\begin{equation}
\mathfrak{S}_{b}\equiv -b< \Im(v) < 0
\label{Sdefb}
\end{equation}
to yield a generalized Master equation:

\begin{thm}
If $F(v)$ is meromorphic in $\mathfrak{S}_{b}$, the integral exists and
\begin{equation}
F \! \left(v \right)+F \! \left(-i\,b -v \right) = 0\,,
\label{Crit}
\end{equation}
then
\begin{equation}
\int_{-\infty}^{\infty}F \! \left(v \right)d v = 
-\pi i  \sum_j\; \mathit{R} \! \left(F \! \left(v_{j} \right)\right)_{v_{ j}\in \mathfrak{S}_{\;b}}\,.
\label{Master}
\end{equation}
\end{thm}

\begin{proof}

Consider the contour integral that encloses $\mathfrak{S}_{b}$

\begin{equation}
J\equiv \,\int_{-\infty}^{\infty}F \! \left(v \right)d v +\int_{\infty -i\,b}^{-\infty -i\,b}F \! \left(v \right)d v\,.
\label{Jtot}
\end{equation}
Following a change of variables, \eqref{Jtot} becomes

\begin{equation}
J=\,\int_{-\infty}^{\infty}\left(F \! \left(v \right)-F \! \left(-i\,b -v \right)\right)d v=\,-2\pi i 
\sum_j\; \mathit{R} \! \left(F \! \left(v_{j} \right)\right)_{v_{ j}\in \mathfrak{S}_{b}}\,,
\label{Jtot2}
\end{equation}
and if \eqref{Crit} is true, then \eqref{Master} immediately follows.
\end{proof}

\begin{cor}

Further, with $h(v)\in\mathfrak{C}$ and continuous and 
\begin{equation}
F \! \left(v \right)-\,F \! \left(-i\,b -v \right) =h(v)\,F \! \left(v \right)
\label{Crit2}
\end{equation}
then
\begin{equation}
\int_{-\infty}^{\infty}F \! \left(v \right)\,h(v)\,d v = 
\,-2\pi i  \sum_j\; \mathit{R} \! \left(F \! \left(v_{j} \right)\right)_{v_{ j}\in \mathfrak{S}_{\;b}}\,.
\label{Master2}
\end{equation}
\end{cor}

\begin{cor}
If $h(v)\in\mathfrak{C}$ is continuous and 
\begin{equation}
F \! \left(v \right)+\,F \! \left(-i\,b -v \right) = h(v)\,F \! \left(v \right)
\label{Crit2a}
\end{equation}
then

\begin{equation}
\int_{-\infty}^{\infty}F \! \left(v \right)(2-h(v))\,d v =\, 
-2\pi i  \sum_j\; \mathit{R} \! \left(F \! \left(v_{j} \right)\right)_{v_{ j}\in \mathfrak{S}_{\;b}}\,
\label{Master2a}
\end{equation}
\end{cor}

\begin{rem} 
In contrast to \eqref{Crit}, the function $h(v)$ introduced in either of \eqref{Crit2} or \eqref{Crit2a} does not impose the existence of a condition, but merely states the obvious -- there always exists $h(v)$ such that 
\begin{equation}
h(v)=1\pm F(-ib-v)/F(v)
\end{equation}
and $h(v)$ becomes of interest if $h(v)\neq 1$ or $h(v)\neq 2$. If $h(v)=0$, then \eqref{Master2a} reduces to \eqref{Master}. See Section \ref{sec:Mult}.
\end{rem}

\begin{rem}
Among other possibilities, to satisfy \eqref{Crit} choose a function $g(v)$ such that

\begin{equation}
F \! \left(v \right) = 
g \! \left(v \right)-g \! \left(-i\,b -v \right)\,;
\label{FgenA}
\end{equation}

to satisfy \eqref{Crit2} choose a function $g(v)$ such that

\begin{equation}
F \! \left(v \right) = 
\frac{g \! \left(a-i\,v \right)\,g \! \left(a-b+i\,v \right)}{\cosh(\pi\,v)}\,,\hspace{10pt} |b|\neq 1\,,
\label{FgenM}
\end{equation}
and to satisfy \eqref{Crit2a} choose a function $g(v)$ such that

\begin{equation}
F \! \left(v \right) = 
g \! \left(a-i\,v \right)\,g \! \left(a+b+i\,v\right)\,.
\label{FgenM2}
\end{equation}
Many other possibilities were discussed in \cite{Master}.
\end{rem}

Here, we consider the application of the generalized Master equation\eqref{Master} to evaluate novel forms of certain integrals over an infinite domain with the emphasis on those whose integrand usually contains Riemann's zeta function as a prototype, with the tacit understanding that the use of other functions follows by example, provided that allowance is made for the different analytic structure of alternate choices.  In particular, see Section \ref{sec:remark1} where a novel technique for the analytic evaluation of an integral is demonstrated and Section \ref{sec:Mult} where variations involving the somewhat arbitrary function $h(v)$ are explored. \newline

Regarding notation, $\gamma$ represents the Euler-Mascheroni constant, $\gamma\,(j)$ are the $j^{th}$ Stieltjes constants found in the expansion of Riemann's zeta function $\zeta(s)$ about $s=1$, $m,n \in \mathbb{Z}$ or $\mathbb{N}$ as specified, subscripts $R$ and $I$ respectively refer to the real and imaginary parts of whatever they are attached to, at times enclosed by $\left[..\right]$ for clarity, $\zeta^{(n)}(s)$ refers to the $n^{th}$ derivative of $\zeta(s)$ and $\delta(v)$ is the Dirac delta function. $\xi(s)$ is Riemann's augmented zeta function defined in \eqref{Xi} below and $\mathcal{E}_{j}$ are Euler numbers. The variables $a,~b$ and $s$ are all real, except where noted. At various times, both computer codes Maple \cite{Maple23} and Mathematica \cite{Math23} were employed in the calculations, each being mentioned where it was deemed necessary. The use of $:=$ means symbolic replacement and $\lceil x\rceil$ refers to the ceiling function (least integer greater than x). A (baker's) dozen-and-a-half of the more interesting results follow:

\begin{itemize}

\item{From Section \ref{sec:specific}, equation \eqref{IntG3}:}
\begin{equation} \nonumber
\int_{-\infty}^{\infty}\frac{\zeta \! \left(\frac{3}{2}-i\,v \right) \zeta \! \left(\frac{1}{2}+i\,v \right)}{\cosh^{3} \! \left(\pi \,v \right)}d v
 = 
-\frac{17\,\pi^{2}}{120}+\frac{\gamma^{2}}{2}+\gamma \! \left(1\right)+\frac{{\gamma \left(3\right)}/{3}+\gamma \,\gamma \! \left(2\right)-\gamma \! \left(1\right)^{2}}{\pi^{2}}\,;
\end{equation}

\item{From Section \ref{sec:remark1}, equation \eqref{IntG5b}:}
\begin{equation} \nonumber
\int_{0}^{\infty}\frac{\zeta_{R} \! \left(\frac{3}{2}+i\,v \right)}{\cosh \! \left(\pi \,v \right)}d v
 = 1;
\end{equation}

\item{From Section \ref{sec:SpeCases1}, equation \eqref{K1d}}
\begin{equation} \nonumber
\int_{-\infty}^{\infty}\frac{\xi \! \left(\frac{3}{2}-i\,v \right)}{\left(\frac{3}{2}-i\,v \right) \Gamma \! \left(\frac{3}{4}+\frac{i\,v}{2}\right) \cosh \! \left(\pi \,v \right)}d v
 = \frac{2\,\gamma-1}{\sqrt{\pi}}\,;
\end{equation}

\item{From Section \ref{sec:SpeCases1a}, equation \eqref{Ex6ApB}:}
\begin{equation} \nonumber
\int_{-\infty}^{\infty}\frac{\cosh \! \left(\zeta \! \left(-\frac{b}{2}+i\,v \right)\right)+\cosh \! \left(\zeta \! \left(\frac{b}{2}-i\,v \right)\right)}{\cosh \! \left(\frac{\pi \,v}{b}\right)}d v
 = 2\,\left|b\right|\,\cosh \! \left({1}/{2}\right);
\end{equation}

\item{From Section \ref{sec:MoreGen}, equation \eqref{DiffSq}:}
\begin{align} \nonumber
\int_{-\infty}^{\infty}\frac{\left(\zeta \! \left(a+i\,v  \right)-\zeta \! \left( a +b-i\,v \right)\right)^{2}}{\cosh \! \left(\frac{\pi \,v}{b}\right)}d v
 = 
-\frac{4\,\gamma \,\pi}{\cos \! \left(\frac{\pi  \left(a -1\right)}{b}\right)}+\frac{2\,\pi^{2}}{b\,}\,\frac{\sin \! \left(\frac{\pi  \left(a -1\right)}{b}\right)}{\cos \! \left(\frac{\pi  \left(a -1\right)}{b}\right)^{2}}+\frac{4\, \pi\zeta \! \left(2\,a -1+b \right)}{\cos \! \left(\frac{\pi  \left(a -1\right)}{b}\right)}\,;
\end{align}

\item{From Section \ref{sec:SpeCases2}, equation \eqref{ScMin}:}

\begin{equation} \nonumber
\int_{-\infty}^{\infty}\frac{\left(\zeta \! \left(1-\frac{b}{2}+i\,v \right)-\zeta \! \left(1+\frac{b}{2}-i\,v \right)\right)^{2}}{\cosh \! \left(\frac{\pi \,v}{b}\right)}d v
 = -8\,b\,\gamma \! \left(1\right)+\frac{2\,\pi^{2}}{3\,b}\,;
\end{equation}

\item{From Section \ref{sec:SpCases3}, equation \eqref{J1ab}:}

\begin{equation} \nonumber
\int_{0}^{\infty}\frac{ \zeta_{R} \! \left(1+i\,v \right)-\zeta_{R} \! \left(i\,v \right)}{\cosh \! \left(2\,\pi \,v \right)}d v
 = 
\left({\zeta \! \left({3}/{4}\right)}-{\zeta \! \left({1}/{4}\right)+{\pi}}\right)/{2}\,;
\end{equation}

\item{From Section \ref{sec:CmplxA}, equation \eqref{Zid}:}

\begin{equation} \nonumber
\zeta \! \left({1}/{2}+i\,t \right) \approx 
\frac{1}{2}\overset{N}{\underset{j =0}{\sum}}\; \frac{\mathcal{E}_{\,2j} \left(\zeta^{\left(2\,j \right)}\left(i\,t \right)+\zeta^{\left(2\,j \right)}\left(1+i\,t \right)\right)}{2^{2\,j}\,\Gamma \left(2\,j +1\right)}-\frac{2}{\pi \,\cosh \! \left(\pi \,t \right)}\,;
\end{equation}

\item{From Section \ref{sec:VarCrit1}, equation \eqref{Ct4bm1}:}

\begin{equation} \nonumber
\int_{-\infty}^{\infty}\left[\left(2\,\pi \right)^{i\,v}\,\zeta \! \left(-i\,v \right)^{2}\, \Gamma \! \left(-i\,v \right)\right]_{R}\;\frac{\cosh \! \left(\frac{\pi \,v}{2}\right)}{\cosh \! \left(\pi \,v \right)}\,d v
 = \frac{\zeta \! \left(\frac{1}{2}\right)^{2}}{2}-\frac{\pi}{4}\,;
\end{equation}

\item{From Section \ref{sec:App}, equation \eqref{CritLim}:}

\begin{equation} \nonumber
\underset{b \rightarrow \frac{1}{\sqrt{2}}}{\mathrm{lim}}\! \cos \! \left(\pi \,b^{2}\right)\int_{-\infty}^{\infty}\frac{ \left[\zeta \! \left(\frac{1}{2}-\frac{b}{2}+i\,v \right) \zeta \! \left(\frac{1}{2}+\frac{b}{2}-i\,v \right)\right]_{R}\, \cosh \! \left(\pi \,b\,v \right)}{\cos^{2}\left(\pi \,b^{2}\right)+\sinh^{2}\left(\pi \,b\,v \right)}d v=0\,;
\end{equation}

\item{From Section \ref{sec:VarCrit2}, equation \eqref{Ct2d}:}

\begin{equation} \nonumber
\int_{-\infty}^{\infty}\frac{\zeta \! \left(\frac{1}{2}-i\,v \right) \zeta \! \left(-\frac{1}{2}+i\,v \right)}{\cosh \! \left(\pi \,v \right)}d v
 = {\frac{1}{4}}\,;
\end{equation}

\item{From Section \ref{sec:Rm1}, \eqref{FintG3}:}

\begin{equation} \nonumber
\int_{0}^{\infty}\Re \! \left(\frac{v +2\,i}{\zeta \! \left(-\frac{9}{2}+i\,v \right) \left(i\,v -4\right) \zeta \! \left(-\frac{1}{2}-i\,v \right) v}\right)d v
 = 
\frac{96\,\pi^{5}}{7\,\zeta \! \left(5\right)}-\frac{64\,\pi^{3}}{\zeta \! \left(3\right)}-\frac{\pi}{4\,\zeta \! \left(-\frac{1}{2}\right) \zeta \! \left(-\frac{9}{2}\right)}\,;
\end{equation}

\item{From Section \ref{sec:NotSoC}, equation \eqref{Jh}:}

\begin{equation} \nonumber
\int_{-\infty}^{\infty}\frac{\zeta \! \left(\frac{1}{2}+i\,v \right)}{\frac{1}{2}+i\,v}d v
 = -\pi\,;
\end{equation}

\item{From Section \ref{sec:HurZ}, equation \eqref{J8b}:}
\begin{equation} \nonumber
\int_{-\infty}^{\infty}\frac{\zeta \! \left(\frac{1}{2}+i\,v \right) 2^{i\,v}}{\frac{1}{2}+i\,v}d v
 = -\frac{\pi \,\sqrt{2}}{2}\,;
\end{equation}

\item{From Section \ref{sec:Mult}, equation \eqref{FintBx}:}
\begin{equation} \nonumber
\int_{0}^{\infty}\left(-\frac{\sqrt{\sqrt{a^{2}+v^{2}}-a}}{v}+\frac{v\,\sqrt{\sqrt{v^{2}+\left(a +b \right)^{2}}-a -b}}{b^{2}+v^{2}}+\frac{b\,\sqrt{\sqrt{v^{2}+\left(a +b \right)^{2}}+a +b}}{b^{2}+v^{2}}\right)d v
 = \pi \,{\sqrt{a/2}}\,;
\end{equation}

\item{From Section \ref{sec:Cos}, equation \eqref{CR2b}:}

\begin{equation} \nonumber
\int_{-1}^{1}\left(v^{2}+1\right)^{-1+\frac{s}{2}} \left(1-v^{2}\right)^{-\frac{s}{2}}\,\cos \! \left(s\,\arctan \! \left(v \right)\right)d v
 = 2^{-1+\frac{s}{2}}\,\pi,\hspace{20pt} s<2\,;
\end{equation}

\item{From Section \ref{sec:Sine}, equation \eqref{Test2G}:}

\begin{equation} \nonumber
\int_{-\infty}^{\infty}\frac{\cosh \! \left(2\,\pi \,v \right) \cos \! \left(\pi \,b \right)-1}{\left(\cosh \! \left(2\,\pi \,v \right)-\cos \! \left(\pi \,b \right)\right)^{2}\,\cosh \! \left(\frac{\pi \,v}{b}\right)}\,d v
 = 
-\frac{1+2\,b^{2}}{12\,{| b |}}+\frac{1}{{| b |}}\overset{{\lceil \frac{{| b |}}{2}\rceil}-1}{\underset{j =1}{\sum}}\; \cot \! \left(\frac{j\,\pi}{b}\right) \csc \! \left(\frac{j\,\pi}{b}\right)\,.
\end{equation}

\end{itemize}

\section{Simple examples with a common kernel \texorpdfstring{$1/\cosh(\pi v/b)$}{Lg}} \label{sec:SimpEx}

\subsection{Specific \texorpdfstring{($b=1$)}{Lg}} \label{sec:specific}
Consider some examples, all of which will be focussed on Riemann's Zeta function $\zeta(s)$. Let

\begin{equation}
F(v)=\frac{\zeta \! \left(\frac{3}{2}-i\,v \right)+\zeta \! \left(\frac{1}{2}+i\,v \right)}{\cosh \! \left(\pi \,v \right)}
\label{F1v}
\end{equation}
Then it is easily seen that $F(v)=-F(-v-i)$ and therefore
\begin{equation}
\int_{-\infty}^{\infty}\frac{\zeta \! \left(\frac{3}{2}-i\,v \right)+\zeta \! \left(\frac{1}{2}+i\,v \right)}{\cosh \! \left(\pi \,v \right)}d v
 = 2\,\gamma
\label{Intf1}
\end{equation}
because the only singularity of the integrand in $\mathfrak{S}_{1}$ occurs at $v=-i/2$ with residue $2\,i\gamma /\pi$. Similarly

\begin{equation}
\int_{-{\infty}}^{{\infty}}\frac{\zeta \! \left(\frac{3}{2}-i\,v \right) \zeta \! \left(\frac{1}{2}+i\,v \right)}{\cosh \! \left(\pi \,v \right)}d v
 = 2\,\gamma \! \left(1\right)+\gamma^{2}-\frac{\pi^{2}}{6}\,,
\label{Intf2}
\end{equation}
and
\begin{equation}
\int_{-\infty}^{\infty}\frac{\zeta \! \left(\frac{3}{2}-i\,v \right)^{2}+\zeta \! \left(\frac{1}{2}+i\,v \right)^{2}}{\cosh \! \left(\pi \,v \right)}d v
 = -4\,\gamma \! \left(1\right)+2\,\gamma^{2}+\frac{\pi^{2}}{3}\,.
\label{Intf3}
\end{equation}
Using the same principles we also have

\begin{equation}
\int_{-\infty}^{\infty}\left(\Gamma \! \left({3}/{2}-i\,v \right) \zeta \! \left({1}/{2}+i\,v \right)-\Gamma \! \left({1}/{2}+i\,v \right) \zeta \! \left({3}/{2}-i\,v \right)\right)d v
 = -2\,\pi\,,
\label{Thing1}
\end{equation}

\begin{equation}
\int_{-\infty}^{\infty}\frac{\zeta \! \left(\frac{3}{2}-i\,v \right)-\zeta \! \left(\frac{1}{2}+i\,v \right)}{\cosh ^{2} \left(\pi \,v \right)}d v
 = -\frac{2\,\gamma \! \left(1\right)}{\pi}+\frac{2\,\pi}{3}\,,
\label{IntG1}
\end{equation}

\begin{equation}
\int_{-\infty}^{\infty}\frac{\zeta \! \left(\frac{3}{2}-i\,v \right)^{2}-\zeta \! \left(\frac{1}{2}+i\,v \right)^{2}}{\cosh^{2} \! \left(\pi \,v \right)}d v
 = 
\frac{4\,\gamma \,\pi}{3}-\frac{4\,\gamma \,\gamma \! \left(1\right)-2\,\gamma \! \left(2\right)}{\pi}\,,
\label{IntG2}
\end{equation}
and
\begin{equation}
\int_{-\infty}^{\infty}\frac{\zeta \! \left(\frac{3}{2}-i\,v \right) \zeta \! \left(\frac{1}{2}+i\,v \right)}{\cosh^{3} \! \left(\pi \,v \right)}d v
 = 
-\frac{17\,\pi^{2}}{120}+\frac{\gamma^{2}}{2}+\gamma \! \left(1\right)+\frac{{\gamma \left(3\right)}/{3}+\gamma \,\gamma \! \left(2\right)-\gamma \! \left(1\right)^{2}}{\pi^{2}}\,.
\label{IntG3}
\end{equation}

All of the above can be verified numerically and numerous variations are obvious.
\subsection{Application} \label{sec:remark1}

Elsewhere \cite[Eq. (5.8)]{Mdet}, based on an inverse Mellin transform, it has been shown that 
\begin{equation}
\int_{0}^{\infty}\frac{\zeta_{R} \! \left(\frac{1}{2}+i\,v \right)}{\cosh \! \left(\pi \,v \right)}d v
 =\gamma-1,
 \label{fromDet}
\end{equation} 
so rewriting \eqref{Intf1} as
\begin{equation}
\int_{0}^{\infty}\frac{\zeta_{R} \! \left(\frac{1}{2}+i\,v \right)+\zeta_{R} \! \left(\frac{3}{2}+i\,v \right)}{\cosh \! \left(\pi \,v \right)}d v
 = \gamma
\label{IntG5}
\end{equation}
leads to

\begin{equation}
\int_{0}^{\infty}\frac{\zeta_{R} \! \left(\frac{3}{2}+i\,v \right)}{\cosh \! \left(\pi \,v \right)}d v
 = 1
\label{IntG5b}
\end{equation}
by subtracting the two. 

\begin{rem} \label{Remx} The identity \eqref{IntG5b} can also be independently verified by writing $\zeta_{R} \left(\frac{3}{2}+i\,v \right)$ as a Dirichlet series and interchanging the sum and integration, because both are convergent, yielding

\begin{equation}
\int_{0}^{\infty}\frac{\zeta_{R} \! \left(\frac{3}{2}+i\,v \right)}{\cosh \! \left(\pi \,v \right)}d v
 = 
\overset{\infty}{\underset{j =1}{\sum}}\; {j^{-\frac{3}{2}}}\int_{0}^{\infty}\frac{\cos \left(v\,\ln \left(j \right)\right)}{\cosh \left(\pi \,v \right)}d v\,.
\label{IntG5c1}
\end{equation}

The resulting integral
\begin{equation}
\int_{0}^{\infty}\frac{\cos \! \left(v\,\ln \! \left(j \right)\right)}{\cosh \! \left(\pi \,v \right)}d v
 = \frac{\sqrt{j}}{j +1}
\label{GRx}
\end{equation}
is listed \cite[Eq. 3.981(3)]{G&R} and the remaining sum
\begin{equation}
\overset{\infty}{\underset{j =1}{\sum}}\! \frac{1}{j \left(j +1\right)}
 = 1
\label{IntG5C}
\end{equation}
is well-known. Although the Dirichlet sum corresponding to \eqref{fromDet} does not converge, by reversing the above logic, it is possible to derive \eqref{fromDet} independently, thus demonstrating a novel means of discovering integral identities involving the function $\zeta(s)$ living inside the critical strip $(0<\Re(s)<1)$, by pairing it via \eqref{Crit}, with $\zeta(s_{1})$ when $s_{1}$ resides outside the critical strip, where many convergent summation and integral identities are known.
\end{rem}

\subsection{General} \label{sec:GenX}


We begin with a more general example that satisfies \eqref{InvDef} provided that $p$ is an odd integer:
\begin{equation}
F(v)\equiv\,\left(\frac{f\! \left( a +i\,b\,v\right)^s+f \! \left(a +b-i\,b\,v  \right)^s}{\cosh \! \left(\pi \,v \right)}\right)^{p}\,,
\label{Fgeneral}
\end{equation}
so that, after a simple scaling of integration variables with $p=1$, choosing $f(a+ib)\Rightarrow\zeta(a+ib)$ with $s=n$, we consider the integral
\begin{equation}
\frac{1}{b}\int_{-\infty}^{\infty}\frac{\zeta \left(a+i\,v  \right)^n+\zeta \left(a +b-i\,v  \right)^n}{\cosh \left(\frac{\pi \,v}{b}\right)}d v\,,
\label{IFg1a}
\end{equation}
where $a\in\mathfrak{R}$ and $b\in\mathfrak{R}$. Omitting for now the possibility that the extension $a\in\mathfrak{C}$ and $b\in\mathfrak{C}$ is formally possible (see Section \ref{sec:CmplxA}), it is easily seen that there are three possible residue terms living inside the region $\mathfrak{S}_{b}$, those being delineated by the conditions
\begin{equation}
0<a<1~ {\mathrm{and}}~ a+b>1,
 \label{Condits}
\end{equation}
indexed according to the value of $n$. See Appendix \ref{sec:AppA}. Thus \eqref{Master} becomes
\begin{align} 
\int_{-\infty}^{\infty}&\frac{\zeta \! \left(i\,v +a \right)^n+ \zeta \! \left(-i\,v +a +b \right)^n}{\cosh \! \left(\frac{\pi \,v}{b}\right)}d v 
= 
-i\,\pi\, b/|b| \left(R_{n,1} +\left(R_{n,2} +R_{n,3} \right) \mathrm{H}\! \left(a,b\right) \right)  \,,
\label{Ex6}
\end{align}
where $H(a,b)=1$ if inequalities \eqref{Condits} are true, and $H(a,b)=0$ otherwise, except in the case of equality, when $H(a,b)=1/2$. If $b<0$, the direction of the inequalities \eqref{Condits} reverses and the direction of the contour in \eqref{Jtot} reverses, requiring the use of $\left|b\right|$ on the right-hand side.\newline

\begin{rem} Formally, both $a$ and $b$ can be complex; although this makes it difficult to visualize the meaning of $\mathfrak{S}_{b}$, it simply shifts various pole locations sideways in the complex $v$-plane -- see Sections \ref{sec:CmplxA} and \ref{sec:Rm1}.
\end{rem}

\subsubsection{Special cases: \texorpdfstring{$a=-b/2$, $n=1$, $b>0$}{Lg}}\label{sec:SpeCases1}

In \eqref{Ex6} let $a=-b/2$, replace $b:=2b$ and evaluate the limit of the right-hand side to obtain
\begin{equation}
\int_{0}^{\infty}\frac{ \zeta_{R} \! \left( -b+i\,v \right)+\zeta_{R} \! \left( b-i\,v \right)}{\cosh \! \left(\frac{\pi \,v}{2\,b}\right)}d v
 = -b\,\hspace{20pt} 0<b<1\,.
\label{YLR1}
\end{equation}
Let $b=1/2$ and, after taking \eqref{fromDet} into account, we find
\begin{equation}
\int_{0}^{\infty}\frac{\zeta_{R} \! \left(-\frac{1}{2}+i\,v \right)}{\cosh \! \left(\pi \,v \right)}d v
 = \frac{1}{2}-\gamma\,.
\label{YLR2}
\end{equation}
After applying the functional equation for $\zeta(s)$, \eqref{YLR2} yields the alternative form
\begin{equation}
\Re \;\int_{0}^{\infty}\frac{ {\pi^{i\,v -1}\,\Gamma \left(\frac{3}{4}-\frac{i\,v}{2}\right) \zeta \left(\frac{3}{2}-i\,v \right)}{}}{\cosh \! \left(\pi \,v \right)\Gamma \left(-\frac{1}{4}+\frac{i\,v}{2}\right)}\,d v
 = \frac{1}{2}-\gamma\,,
\label{K1}
\end{equation}
which, in terms of Riemann's functions 
\begin{equation}
\Upsilon \! \left(s \right) \equiv 
\zeta \! \left(s \right) \Gamma \! \left(\frac{s}{2}\right) \pi^{-\frac{s}{2}}
\label{Ups}
\end{equation}
and 
\begin{equation}
\xi \! \left(s \right) \equiv{\,s \left(s -1\right) \Upsilon \! \left(s \right)}/{2}\,,
\label{Xi}
\end{equation}
can alternatively be written

\begin{equation}
\int_{0}^{\infty}\Re \! \left(\frac{\pi^{{i\,v}/{2}}\,\Upsilon \! \left(\frac{3}{2}-i\,v \right)}{\Gamma \! \left(-\frac{1}{4}+\frac{i\,v}{2}\right) \cosh \! \left(\pi \,v \right)}\right)d v
 = {\pi^{{1}/{4}} \left(1/2-\gamma \right)}{}
\label{K1b}
\end{equation}
or

\begin{equation}
\int_{0}^{\infty}{\Re \! \left(\frac{\pi^{{i\,v}/{2}}\,\xi \left(\frac{3}{2}-i\,v \right)}{\left(\frac{3}{2}-i\,v \right) \Gamma \left(\frac{3}{4}+\frac{i\,v}{2}\right)\cosh \! \left(\pi \,v \right)}\right)}{}d v
 = \pi^{{1}/{4}} \left(\gamma-{1}/{2} \right)
\label{K1d}
\end{equation}


\subsubsection{Special case: \texorpdfstring{$a=-b/2$,~ $n>1$: -- additive}{Lg}} {\label{sec:SpeCases1a}}


We now consider the case $n>1$, where the first few residues are listed in Appendix \ref{sec:AppA}. If the inequality \eqref{Condits} is false so that $R_{\,n,2}$ and $R_{\,n,3}$ do not contribute, we have

\begin{equation}
\int_{-\infty}^{\infty}\frac{\zeta \! \left(a+i\,v  \right)^{n}+\zeta \! \left(a +b-i\,v  \right)^{n}}{\cosh \! \left(\frac{\pi \,v}{b}\right)}d v
 = 2\,\left|b\right|\,\zeta \! \left(a +\frac{b}{2}\right)^{n}\,.
\label{Ex6n}
\end{equation}
In the case $a=-b/2$, for $~0<\left|b\right|<2$, 

\begin{equation}
\int_{-\infty}^{\infty}\frac{\zeta \! \left(-\frac{b}{2}+i\,v \right)^{n}+\zeta \! \left(\frac{b}{2}-i\,v \right)^{n}}{\cosh \! \left(\frac{\pi \,v}{b}\right)}d v
 =\, 2\,\left|b\right| \left(-\frac{1}{2}\right)^{n}\,.
\label{Ex6b}
\end{equation}
This leads to a number of curious variations by summing both sides with creatively chosen coefficients, leading, for example, to
\begin{equation}
\int_{-\infty}^{\infty}\frac{{\mathrm e}^{\zeta \left(-\frac{b}{2}+i\,v \right)}+{\mathrm e}^{\zeta \left(\frac{b}{2}-i\,v \right)}}{\cosh \! \left(\frac{\pi \,v}{b}\right)}d v
 =\, 2\,\left|b\right|\,{\mathrm e}^{-{1}/{2}}
\label{Ex6Exp}
\end{equation}
and
\begin{equation}
\int_{-\infty}^{\infty}\frac{{\mathrm e}^{-\zeta \left(-\frac{b}{2}+i\,v \right)}+{\mathrm e}^{-\zeta \left(\frac{b}{2}-i\,v \right)}}{\cosh \! \left(\frac{\pi \,v}{b}\right)}d v
 =\,  2\,\left|b\right|\,{\mathrm e}^{{1}/{2}}\,.
\label{Ex6cB}
\end{equation}

Adding and subtracting \eqref{Ex6Exp} and \eqref{Ex6cB} produces the curious identities

\begin{equation}
\int_{-\infty}^{\infty}\frac{\cosh \! \left(\zeta \! \left(-\frac{b}{2}+i\,v \right)\right)+\cosh \! \left(\zeta \! \left(\frac{b}{2}-i\,v \right)\right)}{\cosh \! \left(\frac{\pi \,v}{b}\right)}d v
 = 2\,\left|b\right|\,\cosh \! \left({1}/{2}\right)
\label{Ex6ApB}
\end{equation}
and

\begin{equation}
\int_{-\infty}^{\infty}\frac{\sinh \! \left(\zeta \! \left(-\frac{b}{2}+i\,v \right)\right)+\sinh \! \left(\zeta \! \left(\frac{b}{2}-i\,v \right)\right)}{\cosh \! \left(\frac{\pi \,v}{b}\right)}d v
 = -2\,\left|b\right|\,\sinh \! \left(\frac{1}{2}\right)\,.
\label{Ex6AmB}
\end{equation}

In the case that $a=-b/2-2m$ with $m>0$, the right-hand side of \eqref{Ex6} vanishes, so, for example, with $b=-2,~n=1$ we find

\begin{equation}
\int_{-\infty}^{\infty}\frac{\zeta \! \left( -1+i\,v\right)+\zeta \! \left(-3\,-i\,v \right)}{\cosh \! \left(\frac{\pi \,v}{2}\right)}d v
 = 0\,.
\label{Scn1}
\end{equation}

\begin{rem} Each of the above obeys the master condition \eqref{Crit} in its own right, and therefore each is valid for any function $\zeta(s):=f(s)$ if the right-hand side and corresponding inequality are replaced by their bespoke equivalents.
\end{rem}

\subsubsection{\texorpdfstring{$n>1$ -- More General}{Lg}}\label{sec:MoreGen}


Similarly, as in Section \ref{sec:SpeCases1a}, we find, if $0<a<1$ and $a+b>1$, then

\begin{equation}
\int_{-\infty}^{\infty}\frac{\zeta \! \left(a+i\,v  \right) \zeta \! \left(a+b-i\,v  \right)}{\cosh \! \left(\frac{\pi \,v}{b}\right)}d v
 = 
\,b\,\zeta \! \left(a +\frac{b}{2}\right)^{2} -\frac{2\,\pi\,\zeta \! \left(2\,a -1+b \right) }{\cos \! \left(\frac{\pi  \left(a -1\right)}{b}\right)}
\label{IFg1d}
\end{equation}
and under the same limitations, with $n=2$, \eqref{Ex6b} becomes
\begin{equation}
\int_{-\infty}^{\infty}\frac{\zeta \! \left(a+i\,v  \right)^{2}+\zeta \! \left(a+b-i\,v \right)^{2}}{\cosh \! \left(\frac{\pi \,v}{b}\right)}d v
 = 
2\,b\,\zeta \! \left(a +\frac{b}{2}\right)^{2} -\frac{4\,\gamma \,\pi}{\cos \! \left(\frac{\pi  \left(a -1\right)}{b}\right)}+\frac{2\,\pi^{2}\,\sin \! \left(\frac{\pi  \left(a -1\right)}{b}\right) }{b\,\cos \! \left(\frac{\pi  \left(a -1\right)}{b}\right)^{2}}\,.
\label{IFg1e}
\end{equation}

Adding and subtracting gives us

\begin{align} \nonumber
\int_{-\infty}^{\infty}&\frac{ \left(\zeta \! \left(a+i\,v  \right)+\zeta \! \left(a+b-i\,v  \right)\right)^{2}}{\cosh \left(\frac{\pi \,v}{b}\right)}\,d v \\&
=4\,b\,\zeta \! \left(a +\frac{b}{2}\right)^{2} -\frac{4\,\gamma \,\pi}{\cos \! \left(\frac{\pi  \left(a -1\right)}{b}\right)} 
 +\frac{2\,\pi^{2}}{b}\frac{\,\sin \! \left(\frac{\pi  \left(a -1\right)}{b}\right)}{\cos \! \left(\frac{\pi  \left(a -1\right)}{b}\right)^{2}}-\frac{4\,\pi\zeta \! \left(2\,a -1+b \right) }{\cos \! \left(\frac{\pi  \left(a -1\right)}{b}\right)}
\label{BothSq}
\end{align}
and

\begin{align} 
\int_{-\infty}^{\infty}\frac{\left(\zeta \! \left(a+i\,v  \right)-\zeta \! \left( a +b-i\,v \right)\right)^{2}}{\cosh \! \left(\frac{\pi \,v}{b}\right)}d v
 = 
-\frac{4\,\gamma \,\pi}{\cos \! \left(\frac{\pi  \left(a -1\right)}{b}\right)}+\frac{2\,\pi^{2}}{b\,}\,\frac{\sin \! \left(\frac{\pi  \left(a -1\right)}{b}\right)}{\cos \! \left(\frac{\pi  \left(a -1\right)}{b}\right)^{2}}+\frac{4\, \pi\zeta \! \left(2\,a -1+b \right)}{\cos \! \left(\frac{\pi  \left(a -1\right)}{b}\right)}\,.
\label{DiffSq}
\end{align}


\subsubsection{Special cases: \texorpdfstring{$a=b/2,~~n=1$}{Lg}}\label{sec:SpeCases2}

In \eqref{Ex6}, let $a=b/2$ and if $2/3<b<2$, we have

\begin{equation}
\int_{0}^{\infty}\frac{\zeta_{R} \! \left({b}/{2}+i\,v \right)+\zeta_{R} \! \left({3\,b}/{2}-i\,v \right)}{\cosh \! \left({\pi \,v}/{b}\right)}d v
 = 
\,-X\frac{\pi}{\sin \! \left({\pi}/{b}\right)}+b\,\zeta \! \left(b \right) \,.
\label{YLR2x}
\end{equation}
Here, $X=1$ except if $b$ lies outside the above limits, in which case $X=0$; however in the case of equality $X=1/2$. If $b=1$, \eqref{YLR2x} reduces to \eqref{IntG5} by evaluating the limit of the right-hand side; if $b=1/2$ then
\begin{equation}
\int_{0}^{\infty}\frac{\zeta \! \left(\frac{1}{4}+i\,v \right)_{R}+\zeta \! \left(\frac{3}{4}-i\,v \right)_{R}}{\cosh \! \left(2\,\pi \,v \right)}\,d v
 = {\zeta \! \left({1}/{2}\right)}/{2}\,.
\label{YLR2b}
\end{equation}

\subsubsection{Special cases: \texorpdfstring{$a=1-b/2$}{Lg}} \label{sec:SpCases2}

In \eqref{BothSq} and \eqref{DiffSq} we take the limit $a\rightarrow 1-b/2$ to find respectively

\begin{equation}
\int_{-\infty}^{\infty}\frac{\left(\zeta \! \left(1-\frac{b}{2}+i\,v \right)+\zeta \! \left(1+\frac{b}{2}-i\,v \right)\right)^{2}}{\cosh \! \left(\frac{\pi \,v}{b}\right)}d v
 = 4\,\gamma^{2}\,b
\label{ScPlus}
\end{equation}
and
\begin{equation}
\int_{-\infty}^{\infty}\frac{\left(\zeta \! \left(1-\frac{b}{2}+i\,v \right)-\zeta \! \left(1+\frac{b}{2}-i\,v \right)\right)^{2}}{\cosh \! \left(\frac{\pi \,v}{b}\right)}d v
 = -8\,b\,\gamma \! \left(1\right)+\frac{2\,\pi^{2}}{3\,b}\,.
\label{ScMin}
\end{equation}

\subsubsection{Special cases: \texorpdfstring{$b=1/2,~n=1$}{Lg}}\label{sec:SpCases3}


Following the same principles as enunciated above with reference to \eqref{Ex6n}, if $b=1/2$, then for particular values of $a$, we have:
\begin{itemize}
\item{if $a=1/2$,}
\begin{equation}
\int_{0}^{\infty}\frac{\zeta_{R} \! \left(\frac{1}{2}+i\,v \right)+\zeta_{R} \! \left(1-i\,v \right)}{\cosh \! \left(2\,\pi \,v \right)}d v
 = ({\zeta \! \left({3}/{4}\right)}+{\pi})/{2}\,;
\label{J1}
\end{equation}
\item{if $a=0$,}
\begin{equation}
\int_{0}^{\infty}\frac{\zeta _{R}\! \left(i\,v \right)+\zeta _{R}\! \left(\frac{1}{2}-i\,v \right)}{\cosh \! \left(2\,\pi \,v \right)}d v
 = {\zeta \! \left({1}/{4}\right)}/{2}\,.
\label{J1a}
\end{equation}
\end{itemize}

Subtracting \eqref{J1} and \eqref{J1a} yields
\begin{equation}
\int_{0}^{\infty}\frac{ \zeta_{R} \! \left(1+i\,v \right)-\zeta_{R} \! \left(i\,v \right)}{\cosh \! \left(2\,\pi \,v \right)}d v
 = 
\left({\zeta \! \left({3}/{4}\right)}-{\zeta \! \left({1}/{4}\right)+{\pi}}\right)/{2}\,
\label{J1ab}
\end{equation}
since $\zeta(s)$ is self-conjugate. 

\begin{rem}  Note that 
\begin{itemize}
\item 
\eqref{J1ab} is convergent because as $v\rightarrow 0$,
\begin{equation}
\zeta \! \left(1+i\,v \right) \approx 
-\frac{i}{v}+\gamma +O\! \left(v \right)
\label{J1ab0}
\end{equation}
and only the imaginary part of the integrand diverges;
\item
\eqref{J1ab} measures the integrated area of $\zeta_{R}(s)$ between the boundaries of the critical strip, weighted by $\mathrm{sech}(2\pi\,v)$ and represents the transformation of \eqref{J1} and \eqref{J1a} into the Cauchy contour integral of the integrand over the boundary of the critical strip.

\end{itemize}
\end{rem}

\subsubsection{Complex \texorpdfstring{$a$}{Lg}} \label{sec:CmplxA}

As noted above, by letting $a\in\mathfrak{C}$, the locations of the poles in the derivation of \eqref{Master} remain inside the region $\mathfrak{S}_{b}$, shifting to the right or left according to the sign of the imaginary part of $a$. However the previous results remain valid. Let us consider the simple case $a=i\,t$ and $b=1$, where $0\leq t\in \mathfrak{R}$, to obtain

\begin{equation}
\zeta \! \left({1}/{2}+i\,t\right) = 
\frac{1}{2}\int_{-\infty}^{\infty}\frac{\zeta \left(i \left(t +v \right)\right)+\zeta \left(1+i\,(t -v) \right)}{\cosh \left(\pi \,v \right)}d v -\frac{\pi}{2\,\cosh \! \left(\pi \,t \right)}\,.
\label{K4x}
\end{equation}
Notice that the integral \eqref{K4x} is asymmetric with respect to $v=0$ and contains a singularity at $v=t$; however as before, the real part is finite, because

\begin{equation}
\underset{v \rightarrow t}{\mathrm{lim}}\! \frac{\zeta \! \left(i \left(t +v \right)\right)+\zeta \! \left(1+i\,(t -v)\right)}{\cosh \! \left(\pi \,v \right)}
 \approx 
\frac{i}{\cosh \! \left(\pi \,t \right) \left(v -t \right)}+O\! \left(1\right)
\label{Klim4}
\end{equation}
showing that the singularity of \eqref{K4x} occurs in the imaginary part of the integrand on the non-negative $v$-axis. If $t=0$, we obtain

\begin{equation}
\int_{0}^{\infty}\frac{\zeta_{R} \! \left(i\,v \right)+\zeta_{R} \! \left(1-i\,v \right)}{\cosh \! \left(\pi \,v \right)}d v
 = \zeta \! \left({1}/{2}\right)+{\pi}/{2}\,,
\label{K4xR1}
\end{equation}

to be compared to \eqref{J1ab}. Similarly, consider the case $a=1/2+it-b/2$, yielding the identity


\begin{align} \nonumber
\zeta \! \left({1}/{2}+i\,t \right)&=
\frac{1}{2\,{| b |}}\int_{-\infty}^{\infty}\frac{\zeta \left(\frac{1-b}{2}+i \left(t +v \right)\right)+\zeta \left(\frac{b +1}{2}+i \left(t-v  \right)\right)}{\cosh \left(\frac{\pi \,v}{b}\right)}d v\\
&-\frac{i\,\pi \,H(a,b)}{\left(\cos \! \left(\frac{\pi}{2\,b}\right) \sinh \! \left(\frac{\pi \,t}{b}\right)+i\,\cosh \! \left(\frac{\pi \,t}{b}\right) \sin \! \left(\frac{\pi}{2\,b}\right)\right) \left|b\right|}\,.
\label{Ex7ab}
\end{align}

Note that \eqref{Ex7ab} is independent of $b$ and particularly it is invariant under the replacement $b:=\, -b$; if $b=\pm 2$, we find
\begin{equation}
\zeta \! \left(1/2+i\,t \right) = 
\frac{1}{4}\int_{-\infty}^{\infty}\frac{\zeta \left(-\frac{1}{2}+i\,t +i\,v \right)+\zeta \left(\frac{3}{2}-i\,v +i\,t \right)}{\cosh \left(\frac{\pi \,v}{2}\right)}d v -\frac{\pi \,\sqrt{2} \left(i\,\sinh \! \left(\frac{\pi \,t}{2}\right)+\cosh \! \left(\frac{\pi \,t}{2}\right)\right)}{2\,\cosh \! \left(\pi \,t \right)}\,,
\label{ZhId}
\end{equation}
giving, in the case $t=0$
\begin{equation}
\int_{-\infty}^{\infty}\frac{\zeta \! \left(-{1}/{2}+2\,i\,v \right)+\zeta \! \left({3}/{2}-2\,i\,v \right)}{\cosh \! \left(\pi \,v \right)}d v
 = 2\,\zeta \! \left({1}/{2}\right)+\pi \,\sqrt{2}\,.
\label{ZhId0}
\end{equation}

Differentiating with respect to $b$ gives, in the case that $b=2$ and $t=0$,

\begin{equation}
\int_{-\infty}^{\infty}\frac{\left(-\zeta^{\left(1\right)}\! \left(-\frac{1}{2}+2\,i\,v \right)+\zeta^{\left(1\right)}\! \left(\frac{3}{2}-2\,i\,v \right)\right) \left(1-2\,i\,v \right)}{\cosh \! \left(\pi \,v \right)}d v
 = \frac{\sqrt{2}\,\pi  \left(\pi -4\right)}{4}\,,
\label{Ex7ab0}
\end{equation}
whereas differentiating with respect to $t$ in the same limit gives

\begin{equation}
\int_{-\infty}^{\infty}\frac{\zeta^{\left(1\right)}\! \left(-\frac{1}{2}+2\,i\,v \right)+\zeta^{\left(1\right)}\! \left(\frac{3}{2}-2\,i\,v \right)}{\cosh\left(\pi \,v \right)  }\,d v
 = \,
- \sqrt{2}\left(-\zeta \! \left(\frac{1}{2}\right) \left(\frac{\gamma}{2}+\frac{\ln \! \left(8\,\pi \right)}{2}+\frac{\pi}{4}\right) \sqrt{2}-\frac{\pi^{2}}{2}\right)\,.
\label{Ex7B}
\end{equation}

\begin{rem} If $b\rightarrow 0$, \eqref{Ex7ab} reduces to a tautology.\end{rem}

We note that, because the weighting function $\rm{sech}(\pi \, \it{v})$ is numerically small for $v\gg 0$, the form of \eqref{K4x} suggests that a reasonable numerical approximation to $\zeta(1/2+i\,t)$ could be obtained by expanding the integrand about $v=0$ and considering cases where $t\gg v>0$ to avoid the singularity, recognizing that the radius of convergence of such an expansion is $v=t$. When this is done, and the resulting integral is analytically evaluated (see \eqref{GRx} and \cite[Eq. (24.7.6)]{NIST}), we obtain

\begin{equation}
\zeta \! \left({1}/{2}+i\,t \right) \approx 
\frac{1}{2}\overset{N}{\underset{j =0}{\sum}}\; \frac{\mathcal{E}_{\,2j} \left(\zeta^{\left(2\,j \right)}\left(i\,t \right)+\zeta^{\left(2\,j \right)}\left(1+i\,t \right)\right)}{2^{2\,j}\,\Gamma \left(2\,j +1\right)}-\frac{2}{\pi \,\cosh \! \left(\pi \,t \right)}
\label{Zid}
\end{equation}
where $\mathcal{E}_{\,2j}$ are Euler numbers and $N \ll t^2$. As a result, \eqref{Zid} asserts that the function $\zeta(\sigma+i\,t)$, at any point on the critical line $\sigma=1/2$  is a (weighted) average of corresponding points on the boundary lines of the critical strip $\sigma=0$ and $\sigma=1$. Notice also that the last (lonely) term in \eqref{Zid} closely approaches zero as $t$ increases, but never theoretically vanishes, although practically it does so (numerically) , lurking in the deep background of any numerical study of the zeros of $\zeta(1/2+it)$.

\subsection{Variations} \label{sec:VarCrit1}

Other possibilities exist. Consider the case \eqref{IFg1d} in more detail, where


\begin{equation}
F(v)\equiv\,\frac{\zeta \! \left(a+i\,v  \right) \zeta \! \left(a +b -i\,v \right)}{\cosh \! \left(\frac{\pi \,v}{b}\right)}
\label{FCrit1}
\end{equation}
obeys the condition \eqref{Crit}. If $b>0$, there are three residues to consider, occurring at the points

\begin{align} \nonumber
&P_{1}=(a-1)i, \\ \nonumber
&P_{2}=(1-a-b)i  ,\\ 
&P_{3}=-ib/2\,,
\label{P123}
\end{align}
with respective residues
\begin{align} \nonumber
&R_{1}=R_{2}=\,-i\,\zeta \! \left(2\,a +b -1\right) \sec \! \left({\pi  \left(a -1\right)}/{b}\right)\,,\\ 
&R_{3}={i\,\zeta \! \left(a +{b}/{2}\right)^{2}\,b}/{\pi}\,.
\label{R123}
\end{align}
$R_{1}$ and $R_{2}$ only contribute if $1-b<a<1$; if $b<0$, $R_{1}$ and $R_{2}$ contribute only if $1<a<1-b$ and the sign of the right-hand side of \eqref{Master} reverses. In the case of equality, each such residue is reduced by half. A simple example follows: if $a=1/2,~b=1$ we reproduce \eqref{Intf2}; similarly if $a=b=1/2$ or $a=1$ and $b=-1/2$ we find

\begin{equation}
\int_{-\infty}^{\infty} \left[\frac{\zeta \! \left(1+i\,v \right) \zeta \! \left(\frac{1}{2}-i\,v \right)}{\cosh \! \left(2\,\pi \,v \right)}\right]_{R}d v
 = 
\zeta \! \left(\frac{1}{2}\right) \pi +{\zeta \! \left(\frac{3}{4}\right)^{2}}/{2}
\label{Ctx2m}
\end{equation}
because both $R_{1}$ and $R_{2}$ are reduced by half in both cases. An interesting variation occurs if we set $a=1/2-b$, yielding
\begin{equation}
\int_{-\infty}^{\infty}\frac{\zeta \! \left(\frac{1}{2}-b +i\,v \right) \zeta \! \left(\frac{1}{2}-i\,v \right)}{\cosh \! \left(\frac{\pi \,v}{b}\right)}d v
 = \,b\,\zeta \! \left(\frac{1}{2}-\frac{b}{2}\right)^{2}
\label{Ctx3}
\end{equation}
and consider the limit $b\rightarrow 0$, producing a tautology by first scaling the integration variable $v:=bv$; this allows us to identify
\begin{equation}
\underset{b \rightarrow 0}{\mathrm{lim}}\; \frac{1}{b\,\cosh \! \left(\frac{\pi \,v}{b}\right)}
 = \delta \! \left(v\right)\,.
\label{CoshLim}
\end{equation}
Finally, we consider the case $a=1/2-b/2$ to find
\begin{align} \nonumber
\int_{-\infty}^{\infty}\frac{\zeta \! \left(\frac{1}{2}-\frac{b}{2}+i\,v \right) \zeta \! \left(\frac{1}{2}+\frac{b}{2}-i\,v \right)}{\cosh \! \left(\frac{\pi \,v}{b}\right)}d v
& = 
2 \int_{0}^{\infty}\frac{ \left[\zeta \! \left(\frac{1}{2}-\frac{b}{2}+i\,v \right) \zeta \! \left(\frac{1}{2}+\frac{b}{2}-i\,v \right)\right]_{\,R}}{\cosh \! \left(\frac{\pi \,v}{b}\right)}\,d v \\
& = 
\,\pm\, b\,\zeta \! \left({1}/{2}\right)^{2}\pm\,X\,\pi \,\sec \! \left(\frac{\pi  \left(-{1}-{b}\right)}{2b}\right)
\label{Ctx4}
\end{align}

where $X=0$ if $|b|<1$, $X=1/2$ if $b=\pm 1$, $X=1$ otherwise, and the overall sign of the right-hand side corresponds to the sign of $b$. This identity can be interpreted as a measurement of the correlation between function values of $\zeta(1/2 +iv)$ on diagonally opposite sides of, and equidistant from, the critical line ($b=0$), relative to the kernel ${\mathrm{sech}}(\pi v/b)$. If $b=1$, \eqref{Ctx4} reduces to \eqref{Ct2e} below, and, taking the functional equation of $\zeta(s)$ into account, we find the alternative

\begin{align} \nonumber
\int_{-\infty}^{\infty}\left[\zeta \! \left(\frac{1}{2}+\frac{b}{2}-i\,v \right)^{2}\Gamma \! \left(\frac{1}{2}+\frac{b}{2}-i\,v \right)\,\frac{\left(2\,\pi \right)^{i\,v}\,\sin \! \left(\pi  \left(\frac{1}{4}-\frac{b}{4}+\frac{i\,v}{2}\right)\right) }{\cosh \! \left(\frac{\pi \,v}{b}\right)}\right]_{R}d v \\
 = 
\frac{\left(2\,\pi \right)^{\frac{1}{2}+\frac{b}{2}}}{2} \left(\zeta \! \left(\frac{1}{2}\right)^{2}\,b +X\,\pi \,\sec \! \left(\frac{\pi  \left(-1-b\right)}{2b}\right)\right)\,.
\label{CTy1a}
\end{align}
Therefore, setting $b=1$ gives
\begin{equation}
\int_{-\infty}^{\infty} \left[\left(2\,\pi \right)^{i\,v}\,\zeta \! \left(1-i\,v\right)^{2}\, \Gamma \! \left(1-i\,v \right)\right]_{I}\;\frac{\sinh \! \left(\frac{\pi \,v}{2}\right)}{\cosh \! \left(\pi \,v \right)}d v
 = 
\pi  \left(\frac{\pi}{2}-\zeta \! \left(\frac{1}{2}\right)^{2}\right)
\label{Ct4b1}
\end{equation}

and if $b=-1$, we find the corresponding identity
\begin{equation}
\int_{-\infty}^{\infty}\left[\left(2\,\pi \right)^{i\,v}\,\zeta \! \left(-i\,v \right)^{2}\, \Gamma \! \left(-i\,v \right)\right]_{R}\;\frac{\cosh \! \left(\frac{\pi \,v}{2}\right)}{\cosh \! \left(\pi \,v \right)}\,d v
 = \frac{\zeta \! \left(\frac{1}{2}\right)^{2}}{2}-\frac{\pi}{4}\,.
\label{Ct4bm1}
\end{equation}

\subsection{Application to \texorpdfstring{\eqref{FgenA}}{Lg}} \label{sec:App}


Here we consider 
\begin{equation}
\zeta \! \left(a+i\,v  \right) \zeta \! \left(a +b -i\,v  \right) \left(\frac{1}{\cosh \! \left(\pi \,b\,v \right)}-\frac{1}{\cosh \! \left(\pi\,b  \left(i\,b +v \right)  \right)}\right)
\label{Frit4}
\end{equation}
which obeys \eqref{Crit} by means of \eqref{FgenA}. Now, there are four residues to consider, occurring at the points

\begin{align} \nonumber
&P_{1}=(a-1)i, \\ \nonumber
&P_{2}=(1-a-b)i  ,\\  \nonumber
&P_{3}=-i/(2b)\,\\
&P_{4}=\frac{i }{2\,b}\left(1-2\,b^{2}\right)
\label{P1234}
\end{align}
with respective (lengthy) residues easily evaluated. In the case $a=1/2, ~ b=1/2$, after considerable elementary, but intricate  simplification, we obtain the identity

\begin{align} \nonumber
\sqrt{2}\int_{-\infty}^{\infty}&\left[\zeta \! \left({1}/{2}+i\,v \right) \zeta \! \left(1-i\,v \right)\right]_{I}\,\frac{ \, \sinh \! \left(\frac{\pi \,v}{2}\right)}{\cosh \! \left(\pi \,v \right)}\,d v
+\sqrt{2} \int_{-\infty}^{\infty}\left[\zeta \! \left({1}/{2}+i\,v \right) \zeta \! \left(1-i\,v \right)\right]_{R}\,\frac{ \cosh \! \left(\frac{\pi \,v}{2}\right)}{\cosh \! \left(\pi \,v \right)}\,d v \\
&-\int_{-\infty}^{\infty} \left[\zeta \! \left({1}/{2}+i\,v \right) \zeta \! \left(1-i\,v \right)\right]_{R}\,\frac{1}{\cosh \! \left(\frac{\pi \,v}{2}\right)}\,d v  
 = \pi \,\zeta \! \left({1}/{2}\right) \left(\sqrt{2}-1\right)\,.
\label{C4f}
\end{align}
Compare with \eqref{Ctx2m}. However, using $a=1/2-b/2$ as in \eqref{Ctx4}, we find that for $-1/\sqrt{2}<b<1/\sqrt{2}$, none of the points defined in \eqref{P1234} lie in the range defined by \eqref{Sdefb}, giving

\begin{equation}
\int_{-\infty}^{\infty}\zeta \! \left(\frac{1}{2}-\frac{b}{2}+i\,v \right) \zeta \! \left(\frac{1}{2}+\frac{b}{2}-i\,v \right) \left(\mathrm{sech}\! \left(\pi \,b\,v \right)-\sec \! \left(\pi \,b   \left(i\,v -b \right)\right)\right)d v
 = 0\,.
\label{Crit4bB}
\end{equation}
If $1/\sqrt{2}<|b|<1$, we find
\begin{align} \nonumber
\int_{-\infty}^{\infty}\zeta \! \left(\frac{1}{2}-\frac{b}{2}+i\,v \right)& \zeta \! \left(\frac{1}{2}+\frac{b}{2}-i\,v \right) \left(\mathrm{sech}\! \left(\pi \,b\,v \right)-\sec \! \left(\pi  \left(i\,v -b \right) b \right)\right)d v \\
 &= 
\frac{2}{{| b |}}\,\zeta \! \left(\frac{b^{2}+b -1}{2\,b}\right) \zeta \! \left(\frac{-b^{2}+b +1}{2\,b}\right)
\label{Crit4bA}
\end{align}
and, if $1<|b|<\sqrt{6}/2$, we obtain

\begin{align} \nonumber
\int_{-\infty}^{\infty}\zeta \! \left(\frac{1}{2}-\frac{b}{2}+i\,v \right)& \zeta \! \left(\frac{1}{2}+\frac{b}{2}-i\,v \right) \left(\mathrm{sech}\! \left(\pi \,b\,v \right)-\sec \! \left(\pi  \left(i\,v -b \right) b \right)\right)d v \\
 &= 
\frac{2\,}{{| b |}}\zeta \! \left(\frac{b^{2}+b -1}{2\,b}\right) \zeta \! \left(\frac{-b^{2}+b +1}{2\,b}\right)+\frac{4\,\pi\,\sin \! \left({\pi \,b^{2}}/{2}\right) \sin \! \left({\pi \,{| b |}}/{2}\right) }{\cos \! \left(\pi \,b \right)+\cos \! \left(\pi \,b^{2}\right)}\,.
\label{Crit4bC}
\end{align}

\begin{rem}
\begin{equation} \nonumber
\end{equation}
\begin{itemize}
\item{}

When $b$ lies in the range $-1/\sqrt{2}<b<1/\sqrt{2}$, \eqref{Crit4bB} demonstrates that the two functions $\zeta(1/2\pm b/2 \mp iv)$ lying on equal but diagonally opposite sides of the critical line $v=1/2$ are orthogonal with respect to integration employing the kernel

\begin{equation}
\frac{1}{\cos \! \left(\pi \,b \left(i\,v -b \right)\right)} = 
\frac{-i\,\sinh \! \left(\pi \,b\,v \right) \sin \! \left(\pi \,b^{2}\right)+\cosh \! \left(\pi \,b\,v \right) \cos \! \left(\pi \,b^{2}\right)}{\left(\sinh^{2}\left(\pi \,b\,v \right)\right) \left(\sin^{2}\left(\pi \,b^{2}\right)\right)+\left(\cosh^{2}\left(\pi \,b\,v \right)\right) \left(\cos^{2}\left(\pi \,b^{2}\right)\right)}
\label{Kernel}
\end{equation}
\item
It is of some interest to rewrite \eqref{Crit4bB} in more transparent form

\begin{align} \nonumber
&\cos \! \left(\pi \,b^{2}\right) \int_{-\infty}^{\infty}\frac{ \left[\zeta \! \left(\frac{1}{2}-\frac{b}{2}+i\,v \right) \zeta \! \left(\frac{1}{2}+\frac{b}{2}-i\,v \right)\right]_{R}\, \cosh \! \left(\pi \,b\,v \right)}{\sin^{2}\left(\pi \,b^{2}\right)-\cosh^{2}\left(\pi \,b\,v\right)}d v\\
 &= 
-\int_{-\infty}^{\infty}\frac{\left[\zeta \! \left(\frac{1}{2}-\frac{b}{2}+i\,v \right) \zeta \! \left(\frac{1}{2}+\frac{b}{2}-i\,v \right)\right]_{R}\,}{\cosh \! \left(\pi b\,v \right)}d v -\sin \! \left(\pi \,b^{2}\right) \int_{-\infty}^{\infty}\frac{ \left[\zeta \! \left(\frac{1}{2}-\frac{b}{2}+i\,v \right) \zeta \! \left(\frac{1}{2}+\frac{b}{2}-i\,v \right)\right]_{I}\, \sinh \! \left(\pi \,b\,v \right)}{\sin^{2}\left(\pi \,b^{2}\right)-\cosh^{2}\left(\pi \,b\,v \right)}d v 
\label{Crit4d}
\end{align}
and inquire about the (obscure) limit of the left-hand side as $b\rightarrow 1/\sqrt{2}$, since the integral diverges, while the coefficient vanishes. From \eqref{Crit4bB} we have

\begin{align} \nonumber
&\underset{b \rightarrow \frac{1}{\sqrt{2}}^{-}}{\mathrm{lim}}\! \cos \! \left(\pi \,b^{2}\right)\int_{-\infty}^{\infty}\frac{ \left[\zeta \! \left(\frac{1}{2}-\frac{b}{2}+i\,v \right) \zeta \! \left(\frac{1}{2}+\frac{b}{2}-i\,v \right)\right]_{R}\, \cosh \! \left(\pi \,b\,v \right)}{-\cos^{2}\left(\pi \,b^{2}\right)-\sinh^{2}\left(\pi \,b\,v \right)}d v  \\
& = 
 \int_{-\infty}^{\infty}\,\frac{ \left[\zeta \! \left(\frac{1}{2}-\frac{\sqrt{2}}{4}+i\,v \right) \zeta \! \left(\frac{1}{2}+\frac{\sqrt{2}}{4}-i\,v \right)\right]_{I}}{\sinh \! \left({\pi \,v}/{\sqrt{2}}\right)}-\frac{ \left[\zeta \! \left(\frac{1}{2}-\frac{\sqrt{2}}{4}+i\,v \right) \zeta \! \left(\frac{1}{2}+\frac{\sqrt{2}}{4}-i\,v \right)\right]_{R}}{\cosh \! \left({\pi \,\,v}/{\sqrt{2}}\right)}\,d v,
\label{Crit4ELb}
\end{align}
and from \eqref{Crit4bA} we have

\begin{align} \nonumber
&\underset{b \rightarrow \frac{1}{\sqrt{2}}^{+}}{\mathrm{lim}}\! \cos \! \left(\pi \,b^{2}\right)\int_{-\infty}^{\infty}\frac{\left[\zeta \! \left(\frac{1}{2}-\frac{b}{2}+i\,v \right) \zeta \! \left(\frac{1}{2}+\frac{b}{2}-i\,v \right)\right]_{R}\, \cosh \! \left(\pi \,b\,v \right)}{-\cos^{2}\left(\pi \,b^{2}\right)-\sinh^{2}\left(\pi \,b\,v \right)}d v  
   \\ \nonumber
&=\underset{b \rightarrow \frac{1}{\sqrt{2}}^{-}}{\mathrm{lim}}\! \cos \! \left(\pi \,b^{2}\right)\int_{-\infty}^{\infty}\frac{ \left[\zeta \! \left(\frac{1}{2}-\frac{b}{2}+i\,v \right) \zeta \! \left(\frac{1}{2}+\frac{b}{2}-i\,v \right)\right]_{R}\, \cosh \! \left(\pi \,b\,v \right)}{-\cos^{2}\left(\pi \,b^{2}\right)-\sinh^{2}\left(\pi \,b\,v \right)}d v  \\
&\hspace{25pt}+ 2\ \sqrt{2}\zeta \! \left(\frac{\left(\sqrt{2}-1\right) \sqrt{2}}{4}\right) \zeta \! \left(\frac{\left(1+\sqrt{2}\right)}{4}\right) ,
\label{Crit4EUb}
\end{align}
from which we deduce, and verify numerically, that 

\begin{align} \nonumber
& \int_{-\infty}^{\infty}\,\frac{ \left[\zeta \! \left(\frac{1}{2}-\frac{\sqrt{2}}{4}+i\,v \right) \zeta \! \left(\frac{1}{2}+\frac{\sqrt{2}}{4}-i\,v \right)\right]_{I}}{\sinh \! \left({\pi \,v}/{\sqrt{2}}\right)}-\frac{ \left[\zeta \! \left(\frac{1}{2}-\frac{\sqrt{2}}{4}+i\,v \right) \zeta \! \left(\frac{1}{2}+\frac{\sqrt{2}}{4}-i\,v \right)\right]_{R}}{\cosh \! \left({\pi \,\,v}/{\sqrt{2}}\right)}\,d v\\
& = 
- \sqrt{2}\zeta \! \left(\frac{\left(\sqrt{2}-1\right) \sqrt{2}}{4}\right) \zeta \! \left(\frac{\left(1+\sqrt{2}\right)}{4}\right)\,,
\label{Ded1}
\end{align}
in which case, following tradition, we define

\begin{equation} 
\underset{b \rightarrow \frac{1}{\sqrt{2}}}{\mathrm{lim}}\! \cos \! \left(\pi \,b^{2}\right)\int_{-\infty}^{\infty}\frac{ \left[\zeta \! \left(\frac{1}{2}-\frac{b}{2}+i\,v \right) \zeta \! \left(\frac{1}{2}+\frac{b}{2}-i\,v \right)\right]_{R}\, \cosh \! \left(\pi \,b\,v \right)}{\cos^{2}\left(\pi \,b^{2}\right)+\sinh^{2}\left(\pi \,b\,v \right)}d v=0\,,
\label{CritLim}
\end{equation}
representing the mid-point of its discontinuous upper and lower limits.

\item{}
For the same reason that the right-hand side of \eqref{Crit4bB} vanishes, it is possible to show that, if $\Re(s)>0$ and $-1/\sqrt{2}<b<1/\sqrt{2}$, then

\begin{equation}
\int_{-\infty}^{\infty}\zeta \! \left(\frac{1}{2}-\frac{b}{2}+i\,v \right)^{n}\zeta \! \left(\frac{1}{2}+\frac{b}{2}-i\,v \right)^{n} \left(\mathrm{sech}\! \left(\pi \,b\,v \right)^{s}-\mathrm{sech} \left(\pi \,b   \left(i\,v +b \right)\right)^{s}\right)d v
 = 0\,.
\label{Crit4bnS}
\end{equation}
\item{If $b=0$, the kernal of \eqref{Crit4bB} vanishes}.

\end{itemize}
\end{rem}

\subsection{Application to \texorpdfstring{\eqref{Crit2}}{Lg}} \label{sec:VarCrit2}


Here we consider 

\begin{equation}
F \! \left(v \right) = 
\frac{f \! \left(a-i\,v  \right)\,f \! \left( a -1+i\,v\right)}{\cosh \! \left(\pi \,v \right)}
\label{FCrit2}
\end{equation}
which obeys both \eqref{Crit} and \eqref{Crit2} giving

\begin{equation}
\int_{-\infty}^{\infty}\frac{\zeta \! \left(a-i\,v  \right) \zeta \! \left(a -1+i\,v \right)}{\cosh \! \left(\pi \,v \right)}d v
 = 
\zeta \! \left(-\frac{1}{2}+a \right)^{2}-2\,\pi \,X\,\zeta \! \left(2\,a -2\right) \sec \! \left(\pi \,a \right)
\label{Ct2}
\end{equation}
where $X=1$ if $1<a<2$, $X=1/2$ in  the case of equality, otherwise zero. If $a=3/2$, \eqref{Ct2} reduces to \eqref{Intf2}; if $a=1/2$ we have

\begin{equation}
\int_{-\infty}^{\infty}\frac{\zeta \! \left(\frac{1}{2}-i\,v \right) \zeta \! \left(-\frac{1}{2}+i\,v \right)}{\cosh \! \left(\pi \,v \right)}d v
 = {\frac{1}{4}}
\label{Ct2d}
\end{equation}
and if $a=1$ we have
\begin{equation}
\int_{-\infty}^{\infty}\frac{ \left[\zeta \! \left(1-i\,v\right) \zeta \! \left(i\,v \right)\right]_{R}}{\cosh \! \left(\pi \,v \right)}d v
 = \zeta \! \left({1}/{2}\right)^{2}-\frac{\pi}{2}\,.
\label{Ct2e}
\end{equation}
\begin{rem} Notice that $X=1/2$ has been used in this case, and only the imaginary part of the integrand diverges.
\end{rem}



\section{Examples of not-so-obvious convergence} \label{sec:NotSoC}
\subsection{Zeta}
Since the integral appears to converge numerically, the use of \eqref{Crit} gives

\begin{equation}
\int_{-\infty}^{\infty}\left(\frac{\zeta \! \left(a+i\,v  \right)}{a+i\,v }-\frac{\zeta \! \left( a +b-i\,v \right)}{a +b-i\,v }\right)d v
 = -2\,\pi\,,\hspace{10pt} 1-b<a<1,
\label{T1}
\end{equation}

suggesting that we consider the special case
\begin{equation}
\int_{-\infty}^{\infty}\left(\frac{\zeta \! \left(\frac{1}{2}+i\,v \right)}{\frac{1}{2}+i\,v}-\frac{\zeta \! \left(\frac{3}{2}-i\,v \right)}{\frac{3}{2}-i\,v}\right)d v
 = -2\,\pi\,.
\label{IntF6}
\end{equation}

Following the method discussed in Remark \ref{Remx}, the second term of the integrand is amenable to analysis, by considering the general form
\begin{equation}
\int_{-\infty}^{\infty}\frac{\zeta \! \left(a +b-i\,v  \right)}{ a +b-i\,v}d v=\pi.
\label{JaJbx}
\end{equation}
valid if $a+b>1$ - see Appendix \ref{sec:AppB}. Hence, if $b=1$ we find

\begin{equation}
\int_{-\infty}^{\infty}\frac{\zeta \! \left(a+i\,v  \right)}{a+i\,v }d v
 = -\pi,\hspace{10pt} 0<a<1,
\label{JA1}
\end{equation}
and if $a=1/2$ we obtain
\begin{equation}
\int_{-\infty}^{\infty}\frac{\zeta \! \left(\frac{1}{2}+i\,v \right)}{\frac{1}{2}+i\,v}d v
 = -\pi\,,
\label{Jh}
\end{equation}
which can also be written as 
\begin{equation}
\int_{0}^{\infty}\frac{2\,v\,\zeta_{I} \! \left(\frac{1}{2}+i\,v \right) +\zeta_{R} \! \left(\frac{1}{2}+i\,v \right)}{\frac{1}{4}+v^{2}}\,d v
 = -\pi\,.
\label{Jh1}
\end{equation}

{\bf Remarks:} 
\begin{itemize}
\item{}
The result \eqref{Jh} can also be obtained by writing \eqref{JaJb} as a contour integral after a change of variables and translating it past the pole at $a=1$. Thus we similarly find
\begin{equation}
\int_{-\infty}^{\infty}\frac{\zeta \! \left(a+i\,v  \right)}{a+i\,v }d v
 = 0,\hspace{10pt} a<0,
\label{JA1m}
\end{equation}
an identify that confounds numerical verification; this suggests that the integral is probably not numerically convergent and only meaningful as a consequence of regularization (see \cite[Appendix C]{Mdet}).
\item{}
Either of the results \eqref{Jh} and/or \eqref{Jh1} are amenable to numerical verification. This is an unexpected result, because it has been commonly observed that $\zeta_{R} \! \left(\frac{1}{2}+i\,v \right)$ is usually greater than zero.
\item{}
The algebraic relationship between $\zeta_{R}(1/2+iv)$ and $\zeta_{I}(1/2+iv)$ in terms of fundamental quantities, can be found in \cite[Eq. 143]{Milgram}.
\end{itemize}
\subsection{Hurwitz Zeta} \label{sec:HurZ}
Here, we consider a special case of the Hurwitz Zeta function
\begin{equation}
F=\frac{\,\zeta \left(\frac{1}{2}+i\,v , w +\frac{1}{2}\right)}{\frac{1}{2}+i\,v}-\frac{\,\zeta \left(\frac{3}{2}-i\,v, w +\frac{1}{2}\right)}{ \frac{3}{2}-i\,v}\,,
\label{Fz}
\end{equation}
which, with $w=0$ leads us to consider
\begin{equation}
F=\frac{ \left(2^{\frac{1}{2}+i\,v}-1\right) \zeta \! \left(\frac{1}{2}+i\,v \right)}{\frac{1}{2}+i\,v}-\frac{\left(2^{\frac{3}{2}-i\,v }-1\right) \zeta \! \left( \frac{3}{2}-i\,v\right)}{ \frac{3}{2}-i\,v}\,,
\label{F1}
\end{equation}
both of which satisfy the criterion \eqref{Crit}. From \eqref{Master} by utilizing the residues at the point $v=-i/2$, we obtain
\begin{equation}
\int_{-\infty}^{\infty}\left(\frac{\left(2^{\frac{1}{2}+i\,v}-1\right) \zeta \! \left(\frac{1}{2}+i\,v \right)}{\frac{1}{2}+i\,v}-\frac{\left(2^{\frac{3}{2}-i\,v}-1\right) \zeta \! \left(\frac{3}{2}-i\,v \right)}{\frac{3}{2}-i\,v}\right)d v
 = -2\,\pi\,,
\label{AA}
\end{equation}
which, after subtracting  \eqref{IntF6}, yields the identity
\begin{equation}
\int_{-\infty}^{\infty}\frac{\zeta \! \left(\frac{1}{2}+i\,v \right) 2^{i\,v}}{\frac{1}{2}+i\,v}d v -\int_{-\infty}^{\infty}\frac{\zeta \! \left(\frac{3}{2}-i\,v \right) 2^{1-i\,v }}{\frac{3}{2}-i\,v}d v 
 = -2\,\pi \,\sqrt{2}\,.
\label{AAE1}
\end{equation}
Employing the same method as presented in Appendix \ref{sec:AppB}, we find
\begin{equation}
\int_{-\infty}^{\infty}\frac{\zeta \! \left(\frac{3}{2}-i\,v \right) 2^{-i\,v}}{\frac{3}{2}-i\,v}d v
 = \frac{3\,\pi \,\sqrt{2}}{4}
\label{J5e6}
\end{equation}
in which case
\begin{equation}
\int_{-\infty}^{\infty}\frac{\zeta \! \left(\frac{1}{2}+i\,v \right) 2^{i\,v}}{\frac{1}{2}+i\,v}d v
 = -\frac{\pi \,\sqrt{2}}{2}\,.
\label{J8b}
\end{equation}
\begin{rem} By treating \eqref{J5e6} as a contour integral and translating it past the pole at $v=-i/2$, it is possible to verify  \eqref{AAE1}.
\end{rem}

\subsection{Multiplicative} \label{sec:Mult}


Here we consider examples of \eqref{Master2a}; specifically, define

\begin{equation}
F\equiv\,\frac{\zeta \! \left(i\,v +a \right)^{r}\,\zeta \! \left(-i\,v +a +b \right)^{r}}{\left(-i\,v +b \right)^{p} \left(-i\,v +a +b \right)^{q}}
\label{Fpqr}
\end{equation}
which obeys \eqref{Crit2a} and therefore \eqref{Master2a} with
\begin{equation}
h(v)=1+\frac{\left(-i\,v +a +b \right)^{q} \left(-i\,v +b \right)^{p}}{\left(i\,v \right)^{p} \left(i\,v +a \right)^{q}}\,
\label{Hex1}
\end{equation}
where $p,\,q$ and $r$ must be chosen such that the integral exists. To make use of \eqref{Fpqr} we require that if $p>0$ then $p\in\mathbb{N}$ and/or similarly for $q$ in order that a residue be defined. Both cases will be explored in the following examples.

\subsubsection{ \texorpdfstring{$r=1$}{Lg}}

There are six possible singularities inside the strip $\mathfrak{S}_b$, those being

\begin{align*}  
&P_{1}=\,i(a-1)\\ 
&P_{2}=\,-i(a+b-1) \\ 
&P_{3}=-ib \\
&P_{4}=\,-i(a+b) \\
&P_{5}= 0 \\
&P_{6}=ia 
\end{align*}
\begin{equation}
\label{SingPts}
\end{equation} 

of which $P_{3}$ and $P_{5}$ residues will always contribute, each reduced by a factor of two since they sit on the boundary. Suppose $p=q=1$, and $b>0$, then 
\begin{align}
\int_{0}^{\infty}&\Re \! \left(\frac{\zeta \! \left( a +b-i\,v \right) \zeta \! \left(a+i\,v  \right) \left(i\,b +2\,v \right)}{\left( b-i\,v \right) \left( a +b-i\,v \right)  \left( a+i\,v \right)\,v}\right)d v &
\label{FintB1}  \\=\,&
 \frac{\pi \,\zeta \! \left(a +b \right) \zeta \! \left(a \right)}{2 \left(a +b \right) a}-\frac{\pi \,\zeta \! \left(2\,a -1+b \right) \left(2\,a +b -2\right)}{\left(a +b -1\right) \left(2\,a -1+b \right) \left(a -1\right)}\, \hspace{10pt} &\mathrm{if } ~1-b<a<1\,;\\
 =\,& \frac{\pi \,\zeta \! \left(a +b \right) \zeta \! \left(a \right)}{2 \left(a +b \right) a}\hspace{10pt} &\mathrm{if } ~a>1~ \mathrm{or} ~a<-b<0\,;\label{Fint1B} \\
 =\,& \frac{\pi \,\zeta \! \left(a +b \right) \zeta \! \left(a \right)}{2 \left(a +b \right) a}+\frac{\pi \,\zeta \! \left(2\,a +b \right)}{2 \left(a +b \right) a} \hspace{10pt} &\mathrm{if } -b<a<0\,.
\label{Fint1C}
\end{align}

For example, if $a=1/4$ and $b=1$ so that additional (full) residues associated with $P_{1}$ and $P_{2}$ contribute; after some simplification, we find the (numerically verifiable) identity

\begin{equation}
\int_{0}^{\infty}\Im \! \left(\frac{\zeta \! \left({5}/{4}-i\,v \right) \left(1-2\,i\,v \right) \zeta \! \left({1}/{4}+i\,v \right)}{\left(1-i\,v \right) \left(5-4\,i\,v \right) \left(1+4\,i\,v \right)\,v }\right)\,d v
 = 
\frac{\pi \,\zeta \! \left(\frac{1}{2}\right)}{3}-\frac{\pi \,\zeta \! \left(\frac{1}{4}\right) \zeta \! \left(\frac{5}{4}\right)}{10}\,,
\label{Fint1}
\end{equation}
noting that the divergence of the integrand at $v=0$ only occurs in the real part. 


\subsubsection{ \texorpdfstring{$r=0$}{Lg}}

With the exception of $P_{1}$ and $P_{2}$, the singular points are the same as presented in \eqref{SingPts} and the general cases are similar. However, because in this case there is no need to separate the real and imaginary parts of either $\zeta(a+iv)$ or $\zeta(a+b-iv)$, it is possible to obtain a (relatively) compact integral representation in terms of real variables. For example, if $p=1,~ q=-1/2$ and $a,b>0$ we find

\begin{equation}
\int_{0}^{\infty}\left(-\frac{\sqrt{\sqrt{a^{2}+v^{2}}-a}}{v}+\frac{v\,\sqrt{\sqrt{v^{2}+\left(a +b \right)^{2}}-a -b}}{b^{2}+v^{2}}+\frac{b\,\sqrt{\sqrt{v^{2}+\left(a +b \right)^{2}}+a +b}}{b^{2}+v^{2}}\right)d v
 = \pi \,{\sqrt{a/2}}\,.
\label{FintBx}
\end{equation}
Notice that \eqref{FintBx} is independent of the variable $b$, and considering the limit $a\rightarrow 0$ yields

\begin{equation}
\int_{0}^{\infty}\left(\frac{1}{\sqrt{v}}-\frac{\sqrt{\sqrt{v^{2}+1}+1}}{\sqrt{v^{2}+1}}\right)d v
 = 0\,,
\label{FintBx2}
\end{equation} 
an identity known to Mathematica\cite{Math23}.


\subsubsection{ \texorpdfstring{$r=\,-1$}{Lg}} \label{sec:Rm1}


An interesting variation is the case $r=-1$ since the integral now involves the zeros of the zeta function. Consider the variation $p=0,~q=1$ and $0<b<a$, yielding the identity

\begin{equation}
\int_{0}^{\infty}\Re \! \left(\frac{ 2\,v+i\,b}{\zeta \! \left(a+i\,v  \right) \left(b-i\,v  \right) \zeta \! \left(a+b-i\,v  \right)\, v}\right)d v
 = \frac{\pi}{2\,\zeta \! \left(a +b \right) \zeta \! \left(a \right)}\,.
\label{FintG1}
\end{equation}
It is interesting to note that, although the poles associated with the non-trivial zeros of the zeta functions do not lie in the region $\mathfrak{S}_{b}$, even if they did, the residues of the two functions of different argument cancel, whether or not one assumes that they lie on the critical line. In this example, the trivial zeros do not lie in $\mathfrak{S}_{b}$. By setting $a=1$, we further obtain

\begin{equation}
\int_{0}^{\infty}\Re \! \left(\frac{2\,v+i\,b}{\zeta \! \left(1+i\,v\right) \left(b-i\,v \right) \zeta \! \left(b+1-i\,v\right)\, v}\right)d v
 = 0 \,\hspace{20pt} b>1,
\label{FintG1a}
\end{equation}
which can be written as the interesting transformation

\begin{equation}
\int_{0}^{\infty}\Re \! \left(\frac{1}{\zeta \! \left(1+i\,v \right) \left(b-i\,v  \right) \zeta \! \left(b+1-i\,v \right)}\right)d v
 = 
\int_{0}^{\infty}\Im \! \left(\frac{1}{v\,\zeta \! \left(1+i\,v \right) \zeta \! \left(b+1-i\,v  \right)}\right)d v\,.
\label{FintG1b}
\end{equation}
Notice that in the case $b=0$ \eqref{FintG1b} reduces to a tautology. Finally, we consider the case where the trivial zeros of the zeta functions contribute to the poles of the integrand. If $b>0$, $-b<a+2n<0$ and $a+2m<1$ where $m,n$ label the trivial zeros of the zeta functions, the general form is
\begin{align} \nonumber
&\int_{0}^{\infty}\Re \! \left(\frac{i\,b +2\,v}{\zeta \! \left(i\,v +a \right) \left(b-i\,v \right) \zeta \! \left(a+b-i\,v \right) v}\right)d v
 = 
\frac{\pi}{2\,\zeta \! \left(a +b \right) \zeta \! \left(a \right)}\\&
-\frac{\pi  \left(a +\frac{b}{2}+2\,n \right)}{\zeta^{\left(1\right)}\! \left(-2\,n \right) \left(a +b +2\,n \right) \zeta \! \left(2\,a +2\,n +b \right) \left(a +2\,n \right)}-\frac{\pi  \left(a +\frac{b}{2}+2\,m \right)}{\zeta^{\left(1\right)}\! \left(-2m \right) \left(a +b +2\,m \right) \zeta \! \left(2\,a +2\,m +b \right) \left(a +2m \right)}.
\label{FintG2}
\end{align}
As a specific example, if $a=-9/2$ and $b=4$ requiring minor evaluation of limits with $m,n=1~\mathrm{and}~2$, then 

\begin{equation}
\int_{0}^{\infty}\Re \! \left(\frac{v +2\,i}{\zeta \! \left(-\frac{9}{2}+i\,v \right) \left(i\,v -4\right) \zeta \! \left(-\frac{1}{2}-i\,v \right) v}\right)d v
 = 
\frac{96\,\pi^{5}}{7\,\zeta \! \left(5\right)}-\frac{64\,\pi^{3}}{\zeta \! \left(3\right)}-\frac{\pi}{4\,\zeta \! \left(-\frac{1}{2}\right) \zeta \! \left(-\frac{9}{2}\right)},
\label{FintG3}
\end{equation}
where the identities $\zeta^{(1)}(-2)=\,-\zeta(3)/(4\pi^{2})$ and $\zeta^{(1)}(-4)=\,3\zeta(5)/(4\pi^{4})$ have been used \cite{WikiPval}.

\subsubsection{r=2}

There are six possible singularities inside the strip $\mathfrak{S}_b$, and, as an example of what can be obtained we present the identity
\begin{align} \nonumber
\int_{0}^{\infty}& \left[\frac{\zeta \! \left(2-i\,v \right)^{2}\,\zeta \! \left(\frac{1}{2}+i\,v \right)^{2} \left(18\,v^{2}+27\,i\,v -32\right) \left(4\,v +3\,i\right)}{\left(3/2-i\,v \right) \left(2-i\,v \right)^{3} \left(\frac{1}{2}+i\,v \right)^{3}\,v}\right]_{R}d v\\
& = 
\frac{32\,\pi }{9} \left(-\frac{\pi^{4}\,\zeta \left(\frac{1}{2}\right)^{2}}{4}+\left(-\frac{23\,\gamma}{3}+\frac{49}{2}\right) \zeta \! \left(\frac{3}{2}\right)^{2}+\frac{23\,\zeta \left(\frac{3}{2}\right) \zeta^{(1)}\left(\frac{3}{2}\right)}{3}\right)
\label{FintAb}
\end{align}
in the case that $r=2,~p=1,~q=3,a=1/2$ and $b=3/2$.

\section{Other functions - Trigonometric}

Here we consider trigonometric functions $sin$ and $cos$ obeying \eqref{Crit}.
\subsection{Cosine} \label{sec:Cos}
We begin by defining
\begin{equation}
F \! \left(v \right) \equiv 
\frac{\cos^{s}\left(\pi  \left(\frac{b}{2}-i\,v \right)\right)}{\cosh \! \left(\frac{\pi \,v}{b}\right)}\,.
\label{C}
\end{equation}
In order that the integral \eqref{Master} converges, we require $s<1/|b|$ and, in order that the residues be well-defined, we require either that $s\geq 0$ or $s=-1,-2\dots$, all the while noting that $F(v)$ is not analytic at $v=0$ when $s>0$, if $|b|=1$, limiting many cases in this section to $b<1$. First of all, if $-1<b<1$ and $1/|b|>s>0$, then there is a single simple pole in the region defined by $\mathfrak{S}_{b}$ leaving us with the identity

\begin{equation}
\int_{-\infty}^{\infty}\frac{\left(\cosh^{2}\left(\pi \,v \right)-\sin^{2}\left(\frac{\pi \,b}{2}\right)\right)^{{s/}{2}}\,\cos \! \left(s\,\arctan \! \left(\tan \! \left(\frac{\pi \,b}{2}\right) \tanh \! \left(\pi \,v \right)\right)\right)}{\cosh \! \left(\frac{\pi \,v}{b}\right)}d v=|b|\,,
\label{CR1}
\end{equation}
which, after a simple change of variables, can be written as

\begin{equation}
\cos^{s}\left({\pi \,b}/{2}\right)\int_{-1}^{1}\left(\frac{1+B^{2}\,v^{2}}{1-v^{2}}\right)^{{s}/{2}}\,\frac{\cos \! \left(s\,\arctan \! \left(B\,v \right)\right)}{\cosh \! \left(\frac{\mathrm{arctanh}\left(v \right)}{b}\right) \left(1-v^{2}\right)}\,d v
 = \pi \,|b|\,,
\label{Cr2a}
\end{equation}
where
\begin{equation}
B\equiv \tan(\pi\,b/2)\,.
\label{Bdef}
\end{equation}

We now consider the case $s<0$, in which case the limitation $|b|<1/s$ is no longer needed. However, $s<0$ introduces the possibility of new poles in the integrand of \eqref{CR1}, in which case we require that $s=-n$ in order that the resulting residues be well--defined analytically. Beginning with the case $b=2$, we consider successively decreasing values of $s=-1,-2\dots$, all of which are individually known to both Maple and Mathematica. From an examination of the first few residues obtained from this study, it was discovered that the coefficients coincide with a listed OEIS \cite{OEIS} sequence, entry number \seqnum{123746}, suggesting that

\begin{equation}
\int_{-\infty}^{\infty}{\mathrm{sech} \! \left(\frac{\pi \,v}{2}\right) \mathrm{sech}^{n}\left(\pi \,v \right)}\,d v
 = 
-2+2\sqrt{2}\;\overset{n -1}{\underset{k =0}{\sum}}\; \left(-\frac{1}{4}\right)^{k} \binom{2\,k}{k}\,,
\label{Cg}
\end{equation} 
an identity that was numerically tested for other values of $s=-n$, yielding a generalization that does not appear to be known to either of the codes cited above. The case $b=4$ grows ever more complex as $s=-n$ increases, but the residues are easily available from either Maple or Mathematica. As examples, we have

\begin{equation}
\int_{-\infty}^{\infty}\mathrm{sech}\! \left(\frac{\pi \,v}{4}\right) \mathrm{sech}^{3}\, \left(\pi \,v \right)d v
 = 4+\frac{\left(-4\,\sqrt{2}-31\right) \sqrt{2-\sqrt{2}}}{8}=4-\frac{\sqrt{1490-497\,\sqrt{2}}}{8}\,,
\label{B4a}
\end{equation}

and

\begin{equation}
\int_{-\infty}^{\infty}\mathrm{sech}\! \left(\frac{\pi \,v}{4}\right) \mathrm{sech}\! \left(\pi \,v \right)^{4}d v
 = 4+\frac{\left(-33\,\sqrt{2}-103\right) \sqrt{2-\sqrt{2}}}{32} = 4-\frac{\sqrt{11978+809\,\sqrt{2}}}{32}\,.
\label{C4s}
\end{equation}
The first equality in both of the above was obtained by the straightforward evaluation of the Maple--derived residues; the second, and equivalent, equality originates courtesy of Mathematica. Other cases of potential interest include the possibility that $b=1/2$, producing

\begin{equation}
\int_{-1}^{1}\left(v^{2}+1\right)^{-1+{s}/{2}} \left(1-v^{2}\right)^{-{s}/{2}}\,\cos \! \left(s\,\arctan \! \left(v \right)\right)d v
 = 2^{-1+{s}/{2}}\,\pi,\hspace{20pt} s<2\,.
\label{CR2b}
\end{equation}

\begin{rem} Potentially interesting special cases of \eqref{Cr2a} with $s>0$ (e.g. $s=1,b=1/2$) and similar cases of \eqref{CR2b} when $s=-2n$ are known to Maple.
\end{rem}

\subsection{Differentiating}

Differentiating \eqref{CR2b} with respect to $s$ yields the identity
\begin{align} \nonumber
\int_{-1}^{1}\left(v^{2}+1\right)^{{s}/{2}-1}& \left(1-v^{2}\right)^{-{s}/{2}} \left(\ln \! \left(\frac{1-v^{2}}{1+v^{2}}\right) \cos \! \left(s\,\arctan \! \left(v \right)\right)+2\,\arctan \! \left(v \right) \sin \! \left(s\,\arctan \! \left(v \right)\right)\right)d v\\
 &= -{\pi \,2^{1-{s}/{2}}\,\ln \! \left(2\right)}
\label{Cr2bd}
\end{align}
whose special cases $s=0$ and $s=-2$ are known to Maple. In the case that $s=1$ we find
\begin{equation}
\int_{-1}^{1}\frac{2\,v\,\arctan \! \left(v \right)  +\ln \! \left(\frac{1-v^{2}}{1+v^{2}}\right)}{\left(v^{2}+1\right) \sqrt{1-v^{2}}\,}d v
 = -\frac{\pi \,\ln \! \left(2\right)}{\sqrt{2}}\,,
\label{Q1}
\end{equation}
and if $s=-1$ we find
\begin{equation}
\int_{-1}^{1}\frac{\sqrt{1-v^{2}} \left(2\,v\,\arctan \! \left(v \right)  +\ln \! \left(\frac{1+v^{2}}{1-v^{2}}\right)\right)}{\left(v^{2}+1\right)^{2}}d v
 = \frac{\pi \,\sqrt{2}\,\ln \! \left(2\right)}{4}\,,
\label{Q2}
\end{equation}
so that by subtracting and adding the two we also finally find
\begin{equation}
\int_{-1}^{1}\frac{\ln \! \left(\frac{1-v^{2}}{1+v^{2}}\right)+2\, v^{3}\,\arctan \! \left(v \right)}{\sqrt{1-v^{2}} \,\left(v^{2}+1\right)^{2}}d v
 = -\frac{3\,\pi \,\sqrt{2}\,\ln \! \left(2\right)}{8}\,,
\label{Q1mQ2}
\end{equation}
and
\begin{equation}
\int_{-1}^{1}\frac{v^{2}\,\ln \! \left(\frac{1-v^{2}}{1+v^{2}}\right)+2\,v\,\arctan \! \left(v \right) }{\sqrt{1-v^{2}}\, \left(1+v^{2}\right)^{2}}d v
 = -\frac{\pi \,\sqrt{2}\,\ln \! \left(2\right)}{8}\,.
\label{Q1pQ2}
\end{equation}

\subsection{Sine} \label{sec:Sine}

Here we consider the case

\begin{equation}
F(v)\equiv \frac{\sin^{s}\left(\pi  \left(\frac{b}{2}-i\,v \right)\right)}{\cosh \! \left(\frac{\pi \,v}{b}\right)}
\label{S(v)}
\end{equation}
which obeys \eqref{Crit} only if $s=\pm\, 2n$\footnote{In the case $s=2n+1$ we find an example of the condition \eqref{Crit2a} with $h(v)=-2$.}. A superficial study of the case $s=2n$ indicated that the residue of the pole located at $v=-i\,b/2$ vanishes, leading to several simple special case integrals all of which are known to Mathematica. Far more interesting is a consideration of the case $s=-2$, yielding the immediate identity

\begin{equation}
\int_{-\infty}^{\infty}\frac{\cosh \! \left(2\,\pi \,v \right) \cos \! \left(\pi \,b \right)-1}{\cosh \! \left(\frac{\pi \,v}{b}\right) \left(\cosh \! \left(2\,\pi \,v \right)-\cos \! \left(\pi \,b \right)\right)^{2}}\,d v
 = -\frac{1+2\,b^{2}}{12\,\left|b\right|}\,,\hspace{20pt} \left|b\right|<2\,,
\label{Sintm2}
\end{equation}
the validity of which is constrained by the itinerant nature of the poles of the integrand as $b$ increases. From the denominator of \eqref{Sintm2}, in addition to the usual pole at $v=-ib/2$ we see additional poles occurring at two sets of points, labelled by
\begin{equation}
P_{1}\equiv v= -i(b/2 \pm\, n)
\label{P1}
\end{equation} 
and
\begin{equation}
P_{2}\equiv v=-i(n-b/2)\,
\label{P2}
\end{equation}
and must choose $n$ in each case such that the poles reside in the region $\mathfrak{S}_{b}$ for particular values of $b$. A careful study of the first few residues as $b$ increases through increasing values of $2n$ suggests the following (numerically tested) identity

\begin{equation}
\int_{-\infty}^{\infty}\frac{\cosh \! \left(2\,\pi \,v \right) \cos \! \left(\pi \,b \right)-1}{\left(\cosh \! \left(2\,\pi \,v \right)-\cos \! \left(\pi \,b \right)\right)^{2}\,\cosh \! \left(\frac{\pi \,v}{b}\right)}\,d v
 = 
-\frac{1+2\,b^{2}}{12\,{| b |}}+\frac{1}{{| b |}}\overset{{\lceil \frac{{| b |}}{2}\rceil}-1}{\underset{j =1}{\sum}}\; \cot \! \left(\frac{j\,\pi}{b}\right) \csc \! \left(\frac{j\,\pi}{b}\right)\,.
\label{Test2G}
\end{equation}

{\bf Remarks:} If $n=0$, the pole located at $P_{1}$ coincides with the usual pole at $v=-ib/2$, all the singularities become multipoles and the residues located at the points labelled by $P_{2}$ vanish since the poles are of even order. Since Maple was unable to correctly calculate the residues corresponding to the points labelled by $P_{1}$, all the residues were calculated by Mathematica, followed by numerical testing. Additionally, the first few special case integrals $b=n/2$ were known to Maple up to $n=7$, where we find from \eqref{Test2G}
\begin{equation}
\int_{-\infty}^{\infty}\mathrm{sech}\! \left(\frac{2\,\pi \,v}{7}\right) \mathrm{sech}^{2}\! \left(2\,\pi \,v \right)d v
 = 
\frac{17}{28}-\frac{2\,\cot \! \left(\frac{2\,\pi}{7}\right) \csc \! \left(\frac{2\,\pi}{7}\right)}{7}\,;
\label{Shalf7}
\end{equation}
an equivalent test of Mathematica for this integral, yielded an exceedingly complicated expression that defied the Mathematica {\bf FullSimplify} command, although it was numerically correct. Finally, we consider the case $s=-4$ to find

\begin{equation}
\int_{-\infty}^{\infty}\frac{\cos \! \left(2\,\pi \,b \right) \cosh \! \left(4\,\pi \,v \right)-4\,\cosh \! \left(2\,\pi \,v \right) \cos \! \left(\pi \,b \right)+3}{\left(\cosh \! \left(2\,\pi \,v \right)-\cos \! \left(\pi \,b \right)\right)^{4}\,\cosh \! \left(\frac{\pi \,v}{b}\right)}\,d v
 = \frac{88\,b^{4}+40\,b^{2}+7}{720\,b^{3}}\hspace{20pt} |b|<2.
\label{Sint4}
\end{equation}
Further exploration of increasing values of $b$ tested both computer codes to practical limits.

\section{Summary} \label{sec:Sum}

Here we have extended Glasser's Master Theorem by introducing one, and and occasionally two, new parameters and demonstrated the utility of this theorem by evaluating (and numerically testing) a number of interesting improper integrals that mostly, but not exclusively, involve Riemann's Zeta function. It appears that the identities so-established could be generalized to a large number of other special functions, also leading to what appears to be a new technique of evaluating integrals by pairing unknown with known partners living in different regimes of analytic convergence. A number of more general possibilities were presented but not pursued -- for example, \eqref{Fgeneral} introduces two parameters $s$ and $p$, but only the simplest values for these parameters were examined. In several places, the parameter $s$, employed as a integer variable can be (numerically) generalized to real, and sometimes complex values, but the analysis is not straightforward. It is suggested that an investigation of non-integer values of $s$ as it arises here would likely be fruitful. Finally, it seems obvious that a study of the generalization of the parameter $a$ into the complex domain as was touched on in Section \ref{sec:CmplxA} is worthy of further study.

\bibliographystyle{unsrt}
\bibliography{c://physics//biblio}

\begin{thebibliography}{10}

\bibitem{GlasserMaster}
M.~L. Glasser.
\newblock A remarkable definite integral, 2013.
\newblock \url{https://arxiv.org/abs/1308.6361v2}.

\bibitem{Knopp1}
Konrad Knopp.
\newblock {\em Theory of Functions, Part I}.
\newblock Dover, 1945.
\newblock translated by Frederick Bagemihl from the Fourth German Edition.

\bibitem{RaMaster}
Amdeberhan T., Espinosa O., Gonzalez I., Harrison M., Moll V.H., and Straub A.
\newblock Ramanujan's master theorem.
\newblock {\em The Ramanujan Journal}, 29, 2012.
\newblock \url{https://doi.org/10.1007/s11139-011-9333-y}.

\bibitem{BradVig}
Zachary~P. Bradshaw and Christophe Vignat.
\newblock An operational calculus generalization of {R}amanujan's master
  theorem.
\newblock 2022.
\newblock \url{https://doi.org/10.48550/arXiv.2209.03399}.

\bibitem{Master}
M.L. Glasser and M.S.Milgram.
\newblock {M}aster {T}heorems for a family of integrals.
\newblock {\em Integral Transforms and Special Functions}, 25:805--820, 2014.
\newblock \url{http://dx.doi.org/10.1080/10652469.2014.924114;
  http://arxiv.org/abs/1403.2281v2}.

\bibitem{2015arXiv150506515G}
M.~L. {Glasser} and M.~{Milgram}.
\newblock {Uncovering functional relationships at zeros with special reference
  to Riemann's Zeta Function}.
\newblock {\em Journal of ClassicalAnalysis}, 12(1):27--53, 2018.
\newblock \url{https://doi:10.7153/jca-2018-12-04}.

\bibitem{ASQ}
Abu-Ghuwaleh M., Saadeh R., and Qazza A.
\newblock General {M}aster theorems of integrals with applications.
\newblock {\em Mathematics}, 2022.
\newblock \url{https://doi.org/10.3390/math10193547}.

\bibitem{CamAb}
John Campbell and Sanjar Abrarov.
\newblock An analogue of {R}amanujan’s {M}aster theorem.
\newblock 2018.
\newblock \url{https://hal.science/hal-01897255}.

\bibitem{SAQK}
Saadeh R., Abu-Ghuwaleh M., Qazza A., and Kuffi E.
\newblock A fundamental criteria to establish general formulas of integrals.
\newblock {\em Journal of Applied Mathematics}, 2022.
\newblock \url{https://doi.org/10.1155/2022/6049367}.

\bibitem{ASQ2}
Abu-Ghuwaleh M., Saadeh R., and Qazza A.
\newblock A novel approach in solving improper integrals.
\newblock {\em Axioms}, 2022.
\newblock \url{https://doi.org/10.3390/axioms11100572 }.

\bibitem{Mdet}
Michael Milgram.
\newblock Determining the indeterminate: On the evaluation of integrals that
  connect {R}iemann's, {H}urwitz' and {D}irichlet's {Z}eta, {E}ta and {B}eta
  functions.
\newblock 2021.
\newblock available from https://arxiv.org/abs/2107.12559.

\bibitem{HuKimKim}
Daeyeoul~Kim Su~Hu and Min-Soo Kim.
\newblock Special values and integral representations for the {H}urwitz-type
  {E}uler {Z}eta functions.
\newblock {\em J. Korean Math. Soc.,}, 55(1):185--210, 2018.
\newblock https://doi.org/10.4134/JKMS.j170110.

\bibitem{Esp&Moll2002}
Olivier Espinosa and Victor~H. Moll.
\newblock {O}n some integrals involving the {H}urwitz {Z}eta function: Part 1.
\newblock {\em The Ramanujan Journal}, 6:159--188, 2002.
\newblock also available from https://doi.org/10.1023/A:1015706300169 or
  https://arxiv.org/0012078v1.

\bibitem{Shpot&Paris}
M.~A. {Shpot} and R.~B. {Paris}.
\newblock Integrals of products of {H}urwitz zeta functions via {F}eynman
  parametrization and two double sums of {R}iemann zeta functions.
\newblock May 2020.
\newblock available from https://arxiv.org/pdf/1609.05658.

\bibitem{Maple23}
Maplesoft, a division of Waterloo Maple Inc., version 2023.
\newblock {\em Maple.}

\bibitem{Math23}
Wolfram Research, Champaign, Illinois.
\newblock {\em Mathematica, version 13.2}, 2023.

\bibitem{G&R}
I.S. Gradshteyn and I.M. Ryzhik.
\newblock {\em Tables of {I}ntegrals, {S}eries and {P}roducts, corrected and
  enlarged Edition}.
\newblock Academic Press, 1980.

\bibitem{NIST}
F.~W.~J. Olver, D.~W. Lozier, R.~F. Boisvert, and C.~W. Clark, editors.
\newblock {\em {NIST Handbook of Mathematical Functions}}.
\newblock Cambridge University Press, New York, NY, 2010.
\newblock Print companion to \cite{NIST:DLMF}.

\bibitem{Milgram}
M.S. Milgram.
\newblock Integral and series representations of {R}iemann's {Z}eta function,
  {D}irichlet's {E}ta function and a medley of related results.
\newblock {\em Journal of Mathematics, Article ID 181724}, 2013.
\newblock http://dx.doi.org/10.1155/2013/181724.

\bibitem{WikiPval}
https://en.wikipedia.org/wiki/Particular\_values\_of\_the\_{R}iemann\_zeta\_function.

\bibitem{OEIS}
The on-line encyclopedia of integer sequences.
\newblock \url{https://oeis.org}.

\bibitem{NIST:DLMF}
{NIST Digital Library of Mathematical Functions}.
\newblock http://dlmf.nist.gov/, Release 1.0.9 of 2014-08-29.

\end{thebibliography}
\appendix
\section{Appendix A} \label{sec:AppA}

The following are the residues for $n\leq 4$ utilized in section \ref{sec:GenX}, labelled as $R_{n,j}$ where $j=1$ labels the residue emanating from the denominator term of \eqref{Ex6} and $j=2,3$ refers to residues generated by each of the two numerator terms. Also, $A\equiv \pi\,(a-1)/b$.

\begin{itemize}
\item{if $b\neq 0$}
\begin{equation}
R_{\,n,1}\equiv \,\frac{2\,i\, b}{\pi}\zeta \! \left(a +\frac{b}{2}\right)^n\,;
\label{Resnn}
\end{equation}
\item{if $n=1$ and $0<a<1$ and $a+b>1$ then}
\begin{equation}
R_{\,1,\,2}=R_{\,1,\,3}\equiv\, -\frac{i}{\cos \! \left(A \right)}\,;
\label{Resn1}
\end{equation}

\item{if $n=2$ and $0<a<1$ and $a+b>1$ then}
\begin{equation}
R_{\,2,\,2}=R_{\,2,\,3}\equiv \,i\left(-\frac{2\,\gamma}{\cos \! \left(A \right)}+\frac{\pi \,\sin \! \left(A \right)}{b\,\cos ^{2}\! \left(A \right)}\right)\,;
\label{Resn2}
\end{equation}

\item{if $n=3$ and $0<a<1$ and $a+b>1$ then}
\begin{equation}
R_{\,3,\,2}=R_{\,3,\,3}\equiv \,i\left(\frac{3\,\gamma \! \left(1\right)-3\,\gamma^{2}+{\pi^{2}}/{(2\,b^{2})}}{\cos \! \left(A \right)}+\frac{3\,\sin \! \left(A \right) \pi \,\gamma}{b\,\cos^{2} \! \left(A \right)}-\frac{\pi^{2}}{b^{2}\,\cos^{3} \! \left(A \right)}\right)\,;
\label{Resn3}
\end{equation}

\item{if $n=4$ and $0<a<1$ and $a+b>1$ then}
\begin{align} \nonumber
R_{\,4,\,2}=R_{\,4,\,3}\equiv\,i &\left(\frac{12\,\gamma \! \left(1\right) \gamma -2\,\gamma \! \left(2\right)-4\,\gamma^{3}+{2\,\gamma \,\pi^{2}}/{b^{2}}}{\cos \! \left(A \right)}-\frac{\pi \,\sin \! \left(A \right) \left(-36\,\gamma^{2}\,b^{2}+24\,\gamma \! \left(1\right) b^{2}+\pi^{2}\right)}{6\,b^{3}\,\cos^{2} \! \left(A \right)}\right. \\& \left.
\hspace{25pt}-\frac{4\,\pi^{2}\,\gamma}{b^{2}\,\cos ^{3}\! \left(A \right)}+\frac{\pi^{3}\,\sin \! \left(A \right)}{b^{3}\,\cos^{4} \! \left(A \right)}\right)\,.
\label{Resn4}
\end{align}
\end{itemize}

\section{Appendix B -- Proof of \texorpdfstring{\eqref{JaJbx}}{Lg} } \label{sec:AppB}
%

By writing $\zeta \! \left(a+b-i\,v  \right) = \overset{\infty}{\underset{j =1}{\sum}}\; {j^{-a-b+i\,v }}$ and interchanging the sum and integration because both are convergent when $a+b>1$, we obtain

\begin{equation}
\int_{-\infty}^{\infty}\frac{\zeta \! \left(a +b-i\,v  \right)}{ a +b-i\,v}d v=\overset{\infty}{\underset{j =1}{\sum}}\! \frac{1}{j^{b +a}}\int_{-\infty}^{\infty}\left(\frac{\cos \left(v\,\ln \left(j \right)\right) \left(b +a \right)}{v^{2}+\left(b +a \right)^{2}}-\frac{\sin \left(v\,\ln \left(j \right)\right) v}{v^{2}+\left(b +a \right)^{2}}\right)d v,
\label{JaJb}
\end{equation}
after noting that the imaginary parts of the integrand vanish by asymmetry. Each of the two integrals can be evaluated as

\begin{equation}
J_{1}\equiv\left(b +a \right)\int_{-\infty}^{\infty}\frac{\cos \! \left(v\,\ln \! \left(j \right)\right)}{v^{2}+\left(b +a \right)^{2}}d v 
 = \pi \,j^{-b -a}
\label{J1cos}
\end{equation}

and

\begin{equation}
J_{2}\equiv \int_{-\infty}^{\infty}\frac{\sin \! \left(v\,\ln \! \left(j \right)\right) v}{v^{2}+\left(b +a \right)^{2}}d v
 = \pi \,j^{-b -a},  \hspace{20pt} j>1,
\label{J1Sin} 
\end{equation}
by recourse to listed integrals \cite[Eqs. 3.723(2) and 3.723(4)]{G&R}. Therefore, from \eqref{JaJb} we find
\begin{equation}
\pi\;  \overset{\infty}{\underset{j =1}{\sum}}\; j^{-2\,a +2\,b}-\pi \; \overset{\infty}{\underset{j =2}{\sum}}\; j^{-2\,a +2\,b}
 = 
\pi \,\zeta \! \left(2(a+b) \right)-\pi  \left(\zeta \! \left(2\,a +2\,b \right)-1\right)=\pi.
\label{J1CmS}
\end{equation}

\end{flushleft}
\end{document}